\documentclass[12pt]{amsart}

\newtheorem{lemma}{Lemma}
\newtheorem{proposition}{Proposition}
\newtheorem{theorem}{Theorem}
\newtheorem{corollary}{Corollary}

\theoremstyle{definition}
\newtheorem{example}{Example}
\newtheorem{remark}{Remark}

\newtheorem{definition}{Definition}

\usepackage{amssymb}
\usepackage{geometry}
\begin{document}

\title[Schubert decomposition for ind-varieties of generalized flags]{Schubert decompositions for ind-varieties of generalized flags}
\author{Lucas Fresse}
\address{Universit\'e de Lorraine, CNRS, Institut \'Elie Cartan de Lorraine, UMR 7502, Vandoeu\-vre-l\`es-Nancy, F-54506, France}
\email{lucas.fresse@univ-lorraine.fr}
\author{Ivan Penkov}
\address{Jacobs University Bremen, Campus Ring 1, 28759 Bremen, Germany}
\email{i.penkov@jacobs-university.de}


\maketitle

\begin{abstract}
Let $\mathbf{G}$ be one of the ind-groups $GL(\infty)$, $O(\infty)$, $Sp(\infty)$ and $\mathbf{P}\subset \mathbf{G}$ be a splitting parabolic ind-subgroup.\ The ind-variety $\mathbf{G}/\mathbf{P}$ has been identified with an ind-variety of generalized flags in \cite{DP}. In the present paper we define a Schubert cell on $\mathbf{G}/\mathbf{P}$ as a $\mathbf{B}$-orbit on $\mathbf{G}/\mathbf{P}$, where $\mathbf{B}$ is any Borel ind-subgroup of $\mathbf{G}$ which intersects $\mathbf{P}$ in a maximal ind-torus. A significant difference with the finite-dimensional case is that in general $\mathbf{B}$ is not conjugate to an ind-subgroup of $\mathbf{P}$, whence $\mathbf{G}/\mathbf{P}$ admits many non-conjugate Schubert decompositions.\ We study the basic properties of the Schubert cells, proving in particular that they are usual finite-dimensional cells or are isomorphic to affine ind-spaces.

We then define Schubert ind-varieties as closures of Schubert cells and study the smoothness of Schubert ind-varieties.\ Our approach to Schubert ind-varieties differs from an earlier approach by H. Salmasian \cite{Salmasian}.
\end{abstract}

\section{Introduction}

If $G$ is a reductive algebraic group, the flag variety $G/B$ is the most important geometric object attached to $G$.\ If $\mathbf{G}$ is a classical ind-group, $\mathbf{G}=GL(\infty),O(\infty),Sp(\infty)$, then there are infinitely many conjugacy classes of splitting Borel subgroups $\mathbf{B}$, and hence there are infinitely many flag ind-varieties $\mathbf{G}/\mathbf{B}$. These smooth ind-varieties have been studied in \cite{DP1,DP,DPW}, and in \cite{DP} each such ind-variety has been described explicitly as the ind-variety of certain generalized flags in the natural representation $V$ of $\mathbf{G}$.
A generalized flag is a chain of subspaces of $V$ satisfying two conditions (see Definition \ref{definition-1}), but notably such a chain is rarely ordered by an ordered subset of $\mathbb{Z}$.

In this paper we undertake a next step in the study of the generalized flag ind-varieties $\mathbf{G}/\mathbf{B}$, and more generally any ind-variety of the form $\mathbf{G}/\mathbf{P}$ where $\mathbf{P}$ is a splitting parabolic subgroup of $\mathbf{G}$. Namely, we define and study the Schubert decompositions of the ind-varieties  $\mathbf{G}/\mathbf{P}$.\ The Schubert decomposition is a key to many classical theorems in the finite-dimensional case, and its role in the study of the geometry of the ind-varieties $\mathbf{G}/\mathbf{P}$ should be equally important. We define the Schubert cells on $\mathbf{G}/\mathbf{P}$ as the $\mathbf{B}$-orbits on $\mathbf{G}/\mathbf{P}$ for any Borel ind-subgroup $\mathbf{B}$ which contains a common splitting maximal ind-torus with $\mathbf{P}$. The essential difference with the finite-dimensional case is that $\mathbf{B}$ is not necessarily conjugate to a Borel subgroup of $\mathbf{P}$.\ This leads to the existence of many non-conjugate Schubert decompositions of a given ind-variety of generalized flags $\mathbf{G}/\mathbf{P}$. We compute the dimensions of the cells of all Schubert decompositions of $\mathbf{G}/\mathbf{P}$ for any splitting Borel subgroup $\mathbf{B}\subset \mathbf{G}$. We also point out the Bruhat decomposition into double cosets of the ind-group $\mathbf{G}$ which results from a Schubert decomposition of $\mathbf{G}/\mathbf{P}$.

In the last part of the paper we study the smoothness of Schubert ind-varieties which we define as closures of Schubert cells. We establish a criterion for smoothness which allows us to conclude that certain known criteria for smoothness of finite-dimensional Schubert varieties pass to the limit at infinity.

In his work \cite{Salmasian}, H.\ Salmasian introduced Schubert ind-subvarieties of $\mathbf{G}/\mathbf{B}$ as arbitrary direct limits of Schubert varieties on finite-dimensional flag subvarieties of $\mathbf{G}/\mathbf{B}$. He showed that such an ind-variety may be singular at all of its points.\ With our definition, which takes into account the natural action of $\mathbf{G}$ on $\mathbf{G}/\mathbf{B}$, a Schubert ind-variety has always a smooth big cell.

\section{Preliminaries}

In what follows $\mathbb{K}$ is an algebraically closed field of characteristic zero.
All varieties and algebraic groups are defined over $\mathbb{K}$. If $A$ is a finite or infinite set, then $|A|$ denotes its cardinality.\ If $A$ is a subset of the linear space $V$, then $\langle A\rangle$ denotes the linear subspace spanned by $A$.

\subsection{Ind-varieties}

\label{section-2-1}

An {\em ind-variety} is the direct limit $\mathbf{X}=\lim\limits_\to X_n$ of a chain of morphisms of algebraic varieties
\begin{equation}
\label{rel1}
X_1\stackrel{\varphi_1}{\to} X_2\stackrel{\varphi_2}{\to}\cdots\stackrel{\varphi_{n-1}}{\to}
X_n\stackrel{\varphi_n}{\to} X_{n+1}\stackrel{\varphi_{n+1}}{\to}\cdots.
\end{equation}
Note that the direct limit of the chain (\ref{rel1}) does not change if we replace the sequence $\{X_n\}_{n\geq 1}$ by a subsequence $\{X_{i_n}\}_{n\geq 1}$ and the morphisms $\varphi_n$ by the compositions $\tilde\varphi_{i_n}:=\varphi_{i_{n+1}-1}\circ\cdots\circ\varphi_{i_n+1}\circ\varphi_{i_n}$. Let $\mathbf{X}'$ be a second ind-variety obtained as the direct limit of a chain
\[
X'_1\stackrel{\varphi'_1}{\to} X'_2\stackrel{\varphi'_2}{\to}\cdots\stackrel{\varphi'_{n-1}}{\to}
X'_n\stackrel{\varphi'_n}{\to} X'_{n+1}\stackrel{\varphi'_{n+1}}{\to}\cdots.
\]
A {\em morphism of ind-varieties} $\mathbf{f}:\mathbf{X}'\to\mathbf{X}$ is a map from $\lim\limits_\to X'_n$ to $\lim\limits_\to X_n$ induced by a collection of morphisms of algebraic varieties $\{f_n:X'_n\to X_{i_n}\}_{n\geq 1}$ such that $\tilde\varphi_{i_n}\circ f_n=f_{n+1}\circ\varphi'_{n}$ for all $n\geq 1$. The {\em identity morphism} $\mathrm{id}_\mathbf{X}$ is a morphism that induces the identity as a set-theoretic map  $\mathbf{X}\to\mathbf{X}$.\ A morphism $\mathbf{f}:\mathbf{X}'\to\mathbf{X}$ is an {\em isomorphism} if there exists a morphism $\mathbf{g}:\mathbf{X}\to\mathbf{X}'$ such that $\mathbf{g}\circ\mathbf{f}=\mathrm{id}_{\mathbf{X}'}$ and $\mathbf{f}\circ\mathbf{g}=\mathrm{id}_{\mathbf{X}}$.

Any ind-variety $\mathbf{X}$ is endowed with a topology by declaring a subset $\mathbf{U}\subset\mathbf{X}$ {\em open} if its inverse image by the natural map $X_m\to\lim\limits_{\to}X_n$ is open for all $m$.
Clearly, any open (resp., closed) (resp., locally closed) subset $\mathbf{Z}$ of $\mathbf{X}$ has a structure of ind-variety induced by the ind-variety structure on $\mathbf{X}$.
We call $\mathbf{Z}$ an {\em ind-subvariety} of $\mathbf{X}$.

In what follows we only consider chains (\ref{rel1}) where the morphisms $\varphi_n$ are inclusions, so that we can write $\mathbf{X}=\bigcup_{n\geq 1} X_n$.
Then the sequence $\{X_n\}_{n\geq 1}$ is called {\em exhaustion} of $\mathbf{X}$.

Let $x\in \mathbf{X}$, so that $x\in X_n$ for $n$ large enough. Let
$\mathfrak{m}_{n,x}\subset \mathcal{O}_{X_n,x}$ be the maximal ideal of
the localization at $x$ of $\mathcal{O}_{X_n}$. For each $k\geq 1$ we have an
epimorphism
\begin{equation}
\label{**}
\varphi_{n,k}:S^k(\mathfrak{m}_{n,x}/\mathfrak{m}_{n,x}^2)\to \mathfrak{m}_{n,x}^k/\mathfrak{m}_{n,x}^{k+1}.
\end{equation}
Note that the point $x$ is smooth in $X_n$ if and only if
$\varphi_{n,k}$ is an isomorphism for all $k$. By taking the
inverse limit, we obtain a map
\[\hat\varphi_k:=\lim_{\leftarrow}\varphi_{n,k}:\lim_{\leftarrow}S^k(\mathfrak{m}_{n,x}/\mathfrak{m}_{n,x}^2)\to \lim_{\leftarrow}\mathfrak{m}_{n,x}^k/\mathfrak{m}_{n,x}^{k+1}\]
which is an epimorphism for all $k$. We say that $x$ is a {\em smooth
point} of $\mathbf{X}$ if and only if $\hat\varphi_k$ is an isomorphism for
all $k$. We say that $x$ is a {\em singular point} otherwise. The notion of smoothness of a point is independent of the choice of exhaustion $\{X_n\}_{n\geq 1}$ of $\mathbf{X}$. We say that $\mathbf{X}$ is {\em smooth} if every point $x\in\mathbf{X}$ is smooth.
As general references on smooth ind-varieties see \cite[Chapter 4]{K} and \cite{S}.

\begin{example}
\label{example1} {\rm (a)} Assume that every variety $X_n$ in the chain (\ref{rel1}) is an affine space, every image $\varphi_n(X_n)$ is an affine subspace of $X_{n+1}$, and $\lim\limits_{n\to\infty}\dim X_n=\infty$. Then, up to isomorphism,
$\mathbf{X}=\lim\limits_{\to}X_n$ is independent of the choice of $\{X_n,\varphi_n\}_{n\geq 1}$ with these properties. We write $\mathbf{X}=\mathbb{A}^\infty$ and call it the {\em infinite-dimensional affine space}.
For instance, $\mathbb{A}^\infty$ admits the exhaustion $\mathbb{A}^\infty=\bigcup_{n\geq 1}\mathbb{A}^n$ where $\mathbb{A}^n$ stands for the $n$-dimensional affine space. The infinite-dimensional affine space $\mathbb{A}^\infty$ is a smooth ind-variety. \\
{\rm (b)} If every variety $X_n$ in the chain (\ref{rel1}) is a projective space, every image $\varphi_n(X_n)$ is a projective subspace of $X_{n+1}$, and $\lim\limits_{n\to\infty}\dim X_n=\infty$, then $\mathbf{X}=\lim\limits_\to X_n$ is independent of the choice of $\{X_n,\varphi_n\}_{n\geq 1}$ with these properties.\ We write $\mathbf{X}=\mathbb{P}^\infty=\bigcup_{n\geq 1}\mathbb{P}^n$ and call $\mathbb{P}^\infty$ the {\em infinite-dimensional projective space}. The infinite-dimensional projective space $\mathbb{P}^\infty$ is also a smooth ind-variety.
\end{example}

A {\em cell decomposition} of an ind-variety $\mathbf{X}$ is a decomposition
$\mathbf{X}=\bigsqcup_{i\in I} \mathbf{X}_i$ into locally closed ind-subvarieties
$\mathbf{X}_i$, each being a finite-dimensional or infinite-dimensional affine space, and such that the
closure of each $\mathbf{X}_i$ in $\mathbf{X}$ is a union of some subsets $\mathbf{X}_j$ ($j\in I$).

\subsection{Ind-groups}

\label{section-2-2}
An {\em ind-group} is an ind-variety $\mathbf{G}$ endowed with a group structure such that the multiplication $\mathbf{G}\times\mathbf{G}\to\mathbf{G}$, $(g,h)\mapsto gh$, and the inversion $\mathbf{G}\to\mathbf{G}$, $g\mapsto g^{-1}$ are morphisms of ind-varieties. A {\em morphism of ind-groups} $\mathbf{f}:\mathbf{G}'\to\mathbf{G}$  is by definition a morphism of groups which is also a morphism of ind-varieties. A {\em closed ind-subgroup} is a subgroup $\mathbf{H}\subset\mathbf{G}$ which is also a closed ind-subvariety.

We only consider {\em locally linear ind-groups}, i.e., ind-groups admitting an exhaustion $\{G_n\}_{n\geq 1}$ by linear algebraic groups.
Moreover, we focus on the {\em classical ind-groups}
$GL(\infty)$, $O(\infty)$, $Sp(\infty)$, which are obtained as subgroups of the group $Aut(V)$ of linear automorphisms of a countable-dimensional vector space $V$:

\begin{itemize}
\item Let $E$ be a basis of $V$. Define $\mathbf{G}(E)$ as the subgroup of elements $g\in Aut(V)$ such that $g(e)=e$ for all but finitely many basis vectors $e\in E$. Given any filtration $E=\bigcup_{n\geq 1}E_n$ of the basis $E$ by finite subsets, we have \begin{equation}
\label{rel3}
\mathbf{G}(E)=\bigcup_{n\geq 1}G(E_n)
\end{equation}
where $G(E_n)$ stands for $GL(\langle E_n\rangle)$. Thus $\mathbf{G}(E)$ is a locally linear ind-group. We also write $\mathbf{G}(E)=GL(\infty)$.
\item Assume that the space $V$ is endowed with a nondegenerate symmetric or skew-symmetric bilinear form $\omega$.\ A basis $E$ of $V$ is called $\omega$-isotropic if it is equipped with an involution $i_E:E\to E$ with at most one fixed point, such that  $\omega(e,e')=0$ for any $e,e'\in E$ unless $e'=i_E(e)$. Given an $\omega$-isotropic basis $E$ of $V$, define $\mathbf{G}^\omega(E)$ as the subgroup of elements $g\in \mathbf{G}(E)$ which preserve the bilinear form $\omega$. If a filtration $E=\bigcup_{n\geq 1}E_n$ of the basis $E$ by $i_E$-stable finite subsets is fixed, we have
\begin{equation}
\label{rel4}
\mathbf{G}^\omega(E)=\bigcup_{n\geq 1}G^\omega(E_n)
\end{equation}
where $G^\omega(E_n)$ stands for the subgroup of elements $g\in G(E_n)$ preserving the restriction of $\omega$. Thereby $\mathbf{G}^\omega(E)$ has a natural structure of locally linear ind-group. We also write $\mathbf{G}^\omega(E)=Sp(\infty)$ when $\omega$ is symplectic, and $\mathbf{G}^\omega(E)=O(\infty)$ when $\omega$ is symmetric.
\end{itemize}

\begin{remark}
\label{remark-1-new}
{\rm (a)}
Note that the group $\mathbf{G}(E)=GL(\infty)$ depends on the choice of the basis $E$. For this reason, in what follows, we prefer the notation $\mathbf{G}(E)$ instead of $GL(\infty)$.

An alternative construction of $GL(\infty)$ is as follows. Note that the dual space $V^*$ is uncountable dimensional. Let $V_*\subset V^*$ be a countable-dimensional subspace such that the pairing $V_*\times V\to\mathbb{K}$ is nondegenerate. Then the group
\[
\begin{array}[t]{r}\mathbf{G}(V,V_*):=\{g\in Aut(V):g(V_*)=V_*\ \mbox{and there are finite-codimensional subspaces} \\ \mbox{of $V$ and $V_*$ fixed pointwise by $g$}\}
\end{array}
\]
is an ind-group isomorphic to $GL(\infty)$. Moreover, we have $\mathbf{G}(V,V_*)=\mathbf{G}(E)$ whenever $V_*$ is spanned by the dual family of $E$. \\
{\rm (b)} The form $\omega$ induces a countable-dimensional subspace $V_*:=\{\omega(v,\cdot):v\in V\}\subset V^*$ of the dual space.
Then the group
\[\mathbf{G}(V,\omega):=\{g\in \mathbf{G}(V,V_*):\mbox{$g$ preserves $\omega$}\}\]
is an ind-subgroup of $\mathbf{G}(V,V_*)$ isomorphic to $Sp(\infty)$ (if $\omega$ is symplectic) or $O(\infty)$ (if $\omega$ is symmetric).
The equality $\mathbf{G}(V,\omega)=\mathbf{G}^\omega(E)$ holds whenever $E$ is an $\omega$-isotropic basis. \\
{\rm (c)} If $\omega$ is symplectic, then the involution $i_E:E\to E$ has no fixed point; the basis $E$ is said to be {\em of type C} in this case. If $\omega$ is symmetric, then the involution $i_E:E\to E$ can have one fixed point, in which case the basis $E$ is said to be {\em of type B}; if $i_E$ has no fixed point,  the basis $E$ is said to be {\em of type D}. Bases of both types B and D exist in $V$ (see \cite[Lemma 4.2]{DP}).
\end{remark}

In the rest of the paper, we fix once and for all a basis $E$ of $V$ and a filtration $E=\bigcup_{n\geq 1}E_n$ by finite subsets. We assume that the basis $E$ is $\omega$-isotropic and that the subsets $E_n$ are $i_E$-stable
whenever the bilinear form $\omega$ is considered.

Moreover, if the form $\omega$ is symmetric, in view of Remark \ref{remark-1-new}\,{\rm (b)}--{\rm (c)} in what follows we assume that the basis $E$ is of type B and that every subset $E_n$ of the filtration contains the fixed point of the involution $i_E$. This convention ensures that the variety of isotropic flags of a given type of each finite-dimensional space $\langle E_n\rangle$ is connected and $G^\omega(E_n)$-homogeneous. Similarly, every $i_E$-stable finite subset of $E$ considered in the sequel is assumed to contain the fixed point of $i_E$.

By $\mathbf{G}$ we denote one of the ind-groups $\mathbf{G}(E)$, $\mathbf{G}^\omega(E)$.

Let $\mathbf{H}$ be the subgroup of elements $g\in \mathbf{G}$ which are diagonal in the basis $E$. Then $\mathbf{H}$ is a closed ind-subgroup of $\mathbf{G}$ called {\em splitting Cartan subgroup}.
A closed ind-subgroup $\mathbf{B}\subset\mathbf{G}$ which contains $\mathbf{H}$
is called {\em splitting Borel subgroup} if it is locally solvable (i.e., every algebraic subgroup $B\subset\mathbf{B}$ is solvable) and is maximal with this property. A\ closed ind-subgroup which contains such a  splitting Borel subgroup $\mathbf{B}$ is called {\em  splitting parabolic subgroup}.
Equivalently,
a closed ind-subgroup $\mathbf{P}$ of $\mathbf{G}$ containing $\mathbf{H}$ is a  splitting parabolic subgroup of $\mathbf{G}$ if and only if $\mathbf{P}\cap G_n$ is a parabolic subgroup of $G_n$ for all $n\geq 1$, where $\mathbf{G}=\bigcup_{n\geq 1}G_n$ is the natural exhaustion
of (\ref{rel3}) or (\ref{rel4}). The quotient $\mathbf{G}/\mathbf{P}=\bigcup_{n\geq 1}G_n/(\mathbf{P}\cap G_n)$ is a {\em locally projective} ind-variety (i.e., an ind-variety exhausted by projective varieties); note however that $\mathbf{G}/\mathbf{P}$ is in general not a {\em projective} ind-variety (i.e., is not isomorphic to a closed ind-subvariety of the infinite-dimensional projective space $\mathbb{P}^\infty$): see \cite[Proposition 7.2]{DP} and \cite[Proposition 15.1]{DPW}.

In \cite{DP} it is shown that the ind-variety $\mathbf{G}/\mathbf{P}$
can be interpreted as an ind-variety of certain generalized flags. This construction is reviewed in the following section.

\section{Ind-varieties of generalized flags}

\label{section-3}

In Section \ref{section-3-1}
we recall from \cite{DP1, DP} the notion of generalized flag and the correspondence between splitting parabolic subgroups $\mathbf{P}$ of $\mathbf{G}(E)$ and $E$-compatible generalized flags $\mathcal{F}$.  We also recall from \cite{DP} the construction of the ind-varieties $\mathbf{Fl}(\mathcal{F},E)$ of generalized flags and their correspondence with homogeneous ind-spaces of the form $\mathbf{G}(E)/\mathbf{P}$.

In Section \ref{section-3-2} we recall from \cite{DP1, DP} the notion of $\omega$-isotropic generalized flags and the construction of the ind-variety $\mathbf{Fl}(\mathcal{F},\omega,E)$ of $\omega$-isotropic generalized flags, as well as the correspondence with splitting parabolic subgroups of $\mathbf{G}^\omega(E)$ and the corresponding homogeneous ind-spaces.

For later use, some technical aspects of the construction of the ind-varieties $\mathbf{Fl}(\mathcal{F},E)$ and $\mathbf{Fl}(\mathcal{F},\omega,E)$ are emphasized in Section \ref{section-3-3}.

\subsection{Ind-variety of generalized flags}

\label{section-3-1}

By {\em chain} of subspaces of $V$ we mean a set of vector subspaces of $V$ which is totally ordered by inclusion.

\begin{definition}[\cite{DP1,DP}]
\label{definition-1}
A {\em generalized flag} is a chain $\mathcal{F}$ of subspaces of $V$ satisfying the following additional conditions:
\begin{itemize}
\item[{\rm (i)}] every $F\in\mathcal{F}$ has an immediate predecessor $F'$ in $\mathcal{F}$ or an immediate successor $F''$ in $\mathcal{F}$;
\item[{\rm (ii)}] for every nonzero vector $v\in V$, there is a pair $(F',F'')$ of consecutive elements of $\mathcal{F}$ such that $v\in F''\setminus F'$.\end{itemize}
\end{definition}

Let $A_\mathcal{F}$ denote the set of pairs $(F',F'')$ of consecutive subspaces $F',F''\in\mathcal{F}$. The set $A_\mathcal{F}$ is totally ordered by the inclusion of pairs.
Given a totally ordered set $(A,\preceq)$,
we denote by $\mathbf{Fl}_A(V)$ the set of generalized flags such that $(A_\mathcal{F},\subset)$ is isomorphic to $(A,\preceq)$.
Equivalently, $\mathbf{Fl}_A(V)$ is the set of generalized flags $\mathcal{F}$ which can be written in the form
\begin{equation}
\label{notation-Fa1}
\mathcal{F}=\{F'_\alpha,F''_\alpha:\alpha\in A\},
\end{equation}
where $F_\alpha',F''_\alpha$ are subspaces of $V$ such that
\begin{equation}
\label{notation-Fa}
\left\{
\begin{array}{l}
\mbox{$F'_\alpha\varsubsetneq F''_\alpha$ for all $\alpha\in A$;}\\[1mm]
\mbox{$F''_\alpha\subset F'_\beta$ whenever $\alpha\prec \beta$ (possibly $F''_\alpha=F'_\beta$);}\\[1mm]
\mbox{$V\setminus \{0\}=\bigsqcup_{\alpha\in A}F''_\alpha\setminus F'_\alpha$.}
\end{array}\right.
\end{equation}

\begin{definition}
\label{definition-3}
Let $L$ be a basis of the space $V$. A generalized flag $\mathcal{F}=\{F'_\alpha,F''_\alpha:\alpha\in A\}\in\mathbf{Fl}_A(V)$ is said to be {\it compatible with $L$} if there is a (necessarily surjective) map $\sigma:L\to A$ such that
\[F'_\alpha=\langle e\in L:\sigma(e)\prec\alpha\rangle,\quad F''_\alpha=\langle e\in L:\sigma(e)\preceq\alpha\rangle\]
for all $\alpha\in A$.
\end{definition}

Every generalized flag admits a compatible basis (see \cite[Proposition 4.1]{DP}).
The group $\mathbf{G}(E)$  acts on generalized flags in a natural way.
Let $\mathbf{H}(E)\subset\mathbf{G}(E)$ be the splitting Cartan subgroup formed by elements diagonal in $E$.
It is easy to see that a generalized flag $\mathcal{F}$ is compatible with $E$ if and only if it is preserved by $\mathbf{H}(E)$.
We denote by $\mathbf{P}_{\mathcal{F}}\subset\mathbf{G}(E)$ the subgroup of all elements which preserve $\mathcal{F}$.

\begin{proposition}[\cite{DP1,DP}]
\label{proposition-1}
{\rm (a)} If $\mathcal{F}$ is a generalized flag compatible with $E$, then $\mathbf{P}_{\mathcal{F}}$ is a splitting parabolic subgroup of $\mathbf{G}(E)$ containing $\mathbf{H}(E)$.  \\
{\rm (b)} The map $\mathcal{F}\mapsto\mathbf{P}_{\mathcal{F}}$ is a bijection between generalized flags compatible with $E$ and splitting parabolic subgroups of $\mathbf{G}(E)$ containing $\mathbf{H}(E)$. \\
{\rm (c)} A splitting parabolic subgroup $\mathbf{P}_\mathcal{F}$ is a splitting Borel subgroup if and only if the generalized flag $\mathcal{F}$ is maximal
(i.e., $\dim F''/F'=1$ for every pair $(F',F'')$ of consecutive elements of $\mathcal{F}$). \end{proposition}

\begin{remark}
Proposition \ref{proposition-1}\,{\rm (c)} can be interpreted as a version of Lie's theorem for the action of any splitting Borel subgroup on the space $V$. A general version of Lie's theorem has been proved by J.~Hennig in \cite{Hennig}.
\end{remark}

\begin{definition}[\cite{DP}]
\label{definition-3-new}
{\rm (a)}
We say that a generalized flag $\mathcal{F}$ is {\it weakly compatible with $E$} if $\mathcal{F}$ is compatible with a basis $L$ of $V$ such that $E\setminus E\cap L$ is a finite set (equivalently $\mathrm{codim}_V\langle E\cap L\rangle$ is finite). \\
{\rm (b)}
Two generalized flags $\mathcal{F},\mathcal{G}$ are said to be {\it $E$-commensurable} if both $\mathcal{F}$ and $\mathcal{G}$ are weakly compatible with $E$, and
there are an isomorphism of ordered sets
$\phi:\mathcal{F}\to\mathcal{G}$ and a finite-dimensional subspace $U\subset V$ such that
\begin{itemize}
\item[{\rm (i)}] $\phi(F)+U=F+U$ for all $F\in\mathcal{F}$,
\item[{\rm (ii)}] $\dim \phi(F)\cap U=\dim F\cap U$ for all $F\in\mathcal{F}$.
\end{itemize}
\end{definition}

\begin{remark}
{\rm (a)} Clearly, if $\mathcal{F},\mathcal{G}$ are $E$-commensurable with respect to a finite-dimensional subspace $U$, then $\mathcal{F},\mathcal{G}$ are $E$-commensurable with respect to any finite-dimensional subspace $U'\subset V$ such that $U'\supset U$. \\
{\rm (b)} $E$-commensurability is an equivalence relation on the set of generalized flags weakly compatible with $E$.
\end{remark}

Let $\mathcal{F}$ be a generalized flag compatible with $E$.
We denote by $\mathbf{Fl}(\mathcal{F},E)$ the set of all generalized flags which are $E$-commensurable with $\mathcal{F}$.

\begin{proposition}[\cite{DP}]
\label{proposition-2}
The set $\mathbf{Fl}(\mathcal{F},E)$
is endowed with a natural structure of ind-variety.
Moreover, this ind-variety is $\mathbf{G}(E)$-homogeneous and the
map $g\mapsto g\mathcal{F}$ induces
an isomorphism of ind-varieties $\mathbf{G}(E)/\mathbf{P}_\mathcal{F}\cong \mathbf{Fl}(\mathcal{F},E)$.

\end{proposition}

\subsection{Ind-variety of isotropic generalized flags}

\label{section-3-2}

In this section we assume that the space $V$ is endowed with a nondegenerate symmetric or skew-symmetric bilinear form $\omega$.
We write $U^\perp$ for the orthogonal subspace of the subspace $U\subset V$ with respect to $\omega$.
We assume that the basis $E$ is $\omega$-isotropic, i.e., endowed with an involution $i_E:E\to E$ with at most one fixed point and such that, for any $e,e'\in E$, $\omega(e,e')=0$ unless $e'=i_E(e)$.

\begin{definition}[\cite{DP1,DP}]
\label{definition-4}
A generalized flag $\mathcal{F}$ is said to be {\em $\omega$-isotropic} if  $F^\perp\in\mathcal{F}$ whenever $F\in\mathcal{F}$, and if the map $F\mapsto F^\perp$ is an involution of $\mathcal{F}$.
\end{definition}

For $\mathcal{F}$ as in Definition \ref{definition-4},
the involution $F\mapsto F^\perp$ is an anti-automorphism of the ordered set $(\mathcal{F},\subset)$, i.e., it reverses the inclusion relation. Moreover, it induces an involutive anti-automorphism $(F'_\alpha,F''_\alpha)\mapsto ((F''_\alpha)^\perp,(F'_\alpha)^\perp)$ of the set $(A_\mathcal{F},\subset)$ of pairs of consecutive subspaces of $\mathcal{F}$.
Given a totally ordered set $(A,\preceq,i_A)$ equipped with an involutive anti-automorphism $i_A:A\to A$, we denote by $\mathbf{Fl}_A^\omega(V)$ the set of generalized flags $\mathcal{F}\in\mathbf{Fl}_A(V)$ (see (\ref{notation-Fa1})--(\ref{notation-Fa}))
which are $\omega$-isotropic and satisfy the condition
\begin{equation}
\label{notation-Fa3}
((F''_\alpha)^\perp,(F'_\alpha)^\perp)=(F'_{i_A(\alpha)},F''_{i_A(\alpha)})\ \mbox{ for all $\alpha\in A$}.
\end{equation}

\begin{remark}
Note that the set $A$ decomposes as
\[A=A_\ell\sqcup A_0\sqcup A_r\]
where $A_\ell=\{\alpha\in A:\alpha\prec i_A(\alpha)\}$,
$A_0=\{\alpha\in A:\alpha= i_A(\alpha)\}$ (formed by at most one element), $A_r=\{\alpha\in A:\alpha\succ i_A(\alpha)\}$,
and the map $i_A$ restricts to bijections $A_\ell\to A_r$ and $A_r\to A_\ell$.

Given any $\mathcal{F}\in\mathbf{Fl}_A^\omega(V)$, we set
 $\mathcal{T}'=\bigcup_{\alpha\in A_\ell} F''_\alpha$
and $\mathcal{T}''=\bigcap_{\alpha\in A_r} F'_\alpha$.
Clearly, $\mathcal{T}'\subset \mathcal{T}''$, moreover it is easy to see that $(\mathcal{T}')^\perp=\mathcal{T}''$. We have either $\mathcal{T}'=\mathcal{T}''$ or $\mathcal{T}'\varsubsetneq\mathcal{T}''$. In the former case, the anti-automorphism $i_A$ has no fixed point, hence $A=A_\ell\sqcup A_r$. In the latter case, the subspaces $\mathcal{T}',\mathcal{T}''$ necessarily belong to $\mathcal{F}$, moreover we have $(\mathcal{T}',\mathcal{T}'')=(F'_{\alpha_0},F''_{\alpha_0})$ where $\alpha_0\in A$ is the unique fixed point of $i_A$; thus $A=A_\ell\sqcup\{\alpha_0\}\sqcup A_r$ in this case.
\end{remark}

The following lemma shows that the notions of compatibility and weak-compatibility with a basis (Definitions \ref{definition-3}--\ref{definition-3-new}) translate in a natural way to the context of $\omega$-isotropic generalized flags and bases.

\begin{lemma}
\label{lemma-1-new}
Let $\mathcal{F}\in\mathbf{Fl}_A^\omega(V)$, with $(A,\preceq,i_A)$ as above. \\
{\rm (a)}
Let $L$ be an $\omega$-isotropic basis with corresponding involution $i_L:L\to L$.
Assume that $\mathcal{F}$ is compatible with $L$ in the sense of Definition \ref{definition-3}, via a surjective map $\sigma:L\to A$.
Then the map $\sigma$ satisfies $\sigma\circ i_L=i_A\circ \sigma$. \\
{\rm (b)}
Assume that $\mathcal{F}$ is weakly compatible with $E$. Then there is an $\omega$-isotropic basis $L$ such that the set $E\setminus E\cap L$ is finite and $\mathcal{F}$ is compatible with $L$.
\end{lemma}

\begin{proof}
{\rm (a)} For every $e\in L$, we have $e\in F''_{\sigma(e)}\setminus F'_{\sigma(e)}$. Then $i_L(e)\in (F'_{\sigma(e)})^\perp\setminus (F''_{\sigma(e)})^\perp$. The definition of $i_A$ yields $\sigma(i_L(e))=i_A(\sigma(e))$. \\
{\rm (b)}
Let $L$ be a basis of $V$ such that $E\setminus E\cap L$ is finite and $\mathcal{F}$ is compatible with $L$.
Take a subset $E'\subset E$ stable by the involution $i_E$, such that $i_E$ has no fixed point in $E'$, $E\setminus E'$ is finite, and $E'\subset E\cap L$. Then $V'':=\langle E\setminus E'\rangle$ is a finite-dimensional space and the restriction of $\omega$ to $V''$ is nondegenerate. The intersections $\mathcal{F}|_{V''}:=\{F\cap V'':F\in\mathcal{F}\}$ form an isotropic flag of $V''$.\ Since $V''$ is finite dimensional, it is routine to find an $\omega$-isotropic basis $E''$ of $V''$ such that $\mathcal{F}|_{V''}$ is compatible with $E''$. Then $E'\cup E''$ is an $\omega$-isotropic basis of $V$, and $\mathcal{F}$ is compatible with $E'\cup E''$.
\end{proof}

The group $\mathbf{G}^\omega(E)$ acts in a natural way on $\omega$-isotropic generalized flags.
Let $\mathbf{H}^\omega(E)\subset\mathbf{G}^\omega(E)$ be the splitting Cartan subgroup formed by elements diagonal in $E$.
An $\omega$-isotropic generalized flag is compatible with the basis $E$ if and only if it is preserved by $\mathbf{H}^\omega(E)$.
Given an $\omega$-isotropic generalized flag $\mathcal{F}$ compatible with $E$,
we denote by $\mathbf{P}_\mathcal{F}^\omega\subset\mathbf{G}^\omega(E)$ the subgroup of all elements which preserve $\mathcal{F}$.
Moreover, we denote by $\mathbf{Fl}(\mathcal{F},\omega,E)$ the set of all $\omega$-isotropic generalized flags which are $E$-commensurable with $\mathcal{F}$.

\begin{proposition}[\cite{DP1,DP}]
\label{proposition-3-new}
{\rm (a)} The map $\mathcal{F}\mapsto
\mathbf{P}_\mathcal{F}^\omega$ is a bijection between $\omega$-isotropic generalized flags compatible with $E$ and splitting parabolic subgroups of $\mathbf{G}^\omega(E)$ containing $\mathbf{H}^\omega(E)$. \\
{\rm (b)} A splitting parabolic subgroup $\mathbf{P}_\mathcal{F}^\omega$ is a splitting Borel subgroup of $\mathbf{G}^\omega(E)$ if and only if the generalized flag $\mathcal{F}$ is maximal. \\
{\rm (c)} The set $\mathbf{Fl}(\mathcal{F},\omega,E)$ is endowed with a natural structure of ind-variety. This ind-variety is $\mathbf{G}^\omega(E)$-homogeneous and the map $g\mapsto g\mathcal{F}$ induces an isomorphism of ind-varieties $\mathbf{G}^\omega(E)/\mathbf{P}^\omega_\mathcal{F}\cong \mathbf{Fl}(\mathcal{F},\omega,E)$.
\end{proposition}

\subsection{Structure of ind-variety on $\mathbf{Fl}(\mathcal{F},E)$ and $\mathbf{Fl}(\mathcal{F},\omega,E)$}

\label{section-3-3}

In this section we present the structure of ind-variety on
$\mathbf{Fl}(\mathcal{F},E)$ and $\mathbf{Fl}(\mathcal{F},\omega,E)$
mentioned in Propositions \ref{proposition-2}--\ref{proposition-3-new}.

We assume that $\mathcal{F}$ is a generalized flag compatible with the basis $E$.
Let $(A,\preceq)$ be a totally ordered set
such that $\mathcal{F}\in\mathbf{Fl}_A(V)$. Hence we can write
$\mathcal{F}=\{F'_\alpha,F''_\alpha:\alpha\in A\}$. Let $\sigma:E\to A$
be the surjective map corresponding to $\mathcal{F}$ in the sense of Definition \ref{definition-3}.

Let $I\subset E$ be a finite subset. The generalized flag $\mathcal{F}$
gives rise to a (finite) flag $\mathcal{F}|_I$ of the finite-dimensional vector space $\langle I\rangle$
by letting
\[\mathcal{F}|_I:=\{F\cap\langle I\rangle:F\in\mathcal{F}\}=
\{F'_\alpha\cap\langle I\rangle,F''_\alpha\cap\langle I\rangle:\alpha\in A\}.\]
Let
\[d'_\alpha:=\dim F'_\alpha\cap\langle I\rangle=|\{e\in I:\sigma(e)\prec \alpha\}|\quad
\mbox{and}\quad
d''_\alpha:=\dim F''_\alpha\cap\langle I\rangle=|\{e\in I:\sigma(e)\preceq \alpha\}|.\]
We denote by $\mathrm{Fl}(\mathcal{F},I)$ the projective variety of flags in the space $\langle I\rangle$ of the form $\{M'_\alpha,M''_\alpha:\alpha\in A\}$
where $M'_\alpha,M''_\alpha\subset\langle I\rangle$ are linear subspaces such that
\[\dim M'_\alpha=d'_\alpha,\ \dim M''_\alpha=d''_\alpha,\ M'_\alpha\subset M''_\alpha\ \mbox{for all $\alpha\in A$, and }M''_\alpha\subset M'_\beta\mbox{ whenever $\alpha\prec\beta$}.\]
If $J\subset E$ is another finite subset such that $I\subset J$, we define an embedding $\iota_{I,J}:\mathrm{Fl}(\mathcal{F},I)\hookrightarrow \mathrm{Fl}(\mathcal{F},J)$, $\{M'_\alpha,M''_\alpha:\alpha\in A\}\mapsto\{N'_\alpha,N''_\alpha:\alpha\in A\}$ by letting
\[N'_\alpha=M'_\alpha\oplus\langle e\in J\setminus I:\sigma(e)\prec\alpha\rangle
\quad\mbox{and}\quad
N''_\alpha=M''_\alpha\oplus\langle e\in J\setminus I:\sigma(e)\preceq\alpha\rangle\ \mbox{ for all $\alpha\in A$}.\]
If we consider a filtration $E=\bigcup_{n\geq 1}E_n$ of the basis $E$ by finite subsets,
then we obtain a chain of morphisms of projective varieties
\begin{equation}
\label{rel8-newnew}
\mathrm{Fl}(\mathcal{F},E_1)
\stackrel{\iota_1}{\hookrightarrow} \mathrm{Fl}(\mathcal{F},E_2)
\stackrel{\iota_2}{\hookrightarrow}\cdots\stackrel{\iota_{n-1}}{\hookrightarrow}
\mathrm{Fl}(\mathcal{F},E_n)\stackrel{\iota_n}{\hookrightarrow} \mathrm{Fl}(\mathcal{F},E_{n+1})\stackrel{\iota_{n+1}}{\hookrightarrow}\cdots
\end{equation}
where $\iota_n:=\iota_{E_n,E_{n+1}}$.

\begin{proposition}[\cite{DP}]
\label{proposition-3-3-3}
The set $\mathbf{Fl}(\mathcal{F},E)$ is the direct limit of the chain of morphisms
(\ref{rel8-newnew}). Hence $\mathbf{Fl}(\mathcal{F},E)$ is endowed with a structure of ind-variety.
Moreover, this structure is independent of the filtration $\{E_n\}_{n\geq 1}$ of the basis $E$.
\end{proposition}

We assume next that the space $V$ is endowed with a nondegenerate symmetric or skew-symmetric bilinear form $\omega$, that the basis $E$ is $\omega$-isotropic with corresponding involution
$i_E:E\to E$, that the ordered set $(A,\preceq)$ is equipped with
an anti-automorphism $i_A:A\to A$,
and that the surjection $\sigma:E\to A$ satisfies $\sigma\circ i_E=i_A\circ\sigma$ so
that $\mathcal{F}$ is an $\omega$-isotropic generalized flag.

Consider an $i_E$-stable finite subset $I\subset E$. Then the restriction
of $\omega$ to the space $\langle I\rangle$ is nondegenerate.
Let $\mathrm{Fl}(\mathcal{F},\omega,I)\subset\mathrm{Fl}(\mathcal{F},I)$ be the (closed) subvariety formed by flags $\{M'_\alpha,M''_\alpha:\alpha\in A\}$ such that
\[((M''_\alpha)^{\perp_I},(M'_\alpha)^{\perp_I})=(M'_{i_A(\alpha)},M''_{i_A(\alpha)})\ \mbox{ for all $\alpha\in A$},\]
where the notation $\perp_I$ stands for orthogonal subspace in the space $(\langle I\rangle,\omega)$.
If $J\subset E$ is another $i_E$-stable finite subset, then the embedding $\iota_{I,J}$ restricts to an embedding
$\iota_{I,J}^\omega:\mathrm{Fl}(\mathcal{F},\omega,I)\hookrightarrow \mathrm{Fl}(\mathcal{F},\omega,J)$.
Consequently, for a filtration $E=\bigcup_{n\geq 1}E_n$ by $i_E$-stable finite subsets, we obtain a chain of morphisms of projective varieties
\begin{equation}
\label{rel9-newnew}
\mathrm{Fl}(\mathcal{F},\omega,E_1)
\stackrel{\iota^\omega_1}{\hookrightarrow} \mathrm{Fl}(\mathcal{F},\omega,E_2)
\stackrel{\iota^\omega_2}{\hookrightarrow}\cdots\stackrel{\iota^\omega_{n-1}}{\hookrightarrow}
\mathrm{Fl}(\mathcal{F},\omega,E_n)\stackrel{\iota^\omega_n}{\hookrightarrow} \mathrm{Fl}(\mathcal{F},\omega,E_{n+1})\stackrel{\iota^\omega_{n+1}}{\hookrightarrow}\cdots
\end{equation}
where $\iota^\omega_n:=\iota^\omega_{E_n,E_{n+1}}$.

\begin{proposition}[\cite{DP}]
\label{proposition-3-3-2}
The set $\mathbf{Fl}(\mathcal{F},\omega,E)$ is the direct limit of the chain of morphisms
(\ref{rel9-newnew}). Hence $\mathbf{Fl}(\mathcal{F},\omega,E)$ is endowed with a structure of ind-variety, independent of the filtration $\{E_n\}_{n\geq 1}$.
Moreover, $\mathbf{Fl}(\mathcal{F},\omega,E)$ is a closed ind-subvariety of $\mathbf{Fl}(\mathcal{F},E)$.
\end{proposition}

\section{Schubert decomposition of $\mathbf{Fl}(\mathcal{F},E)$ and $\mathbf{Fl}(\mathcal{F},\omega,E)$}

\label{section-4}

Let $\mathbf{G}$ be one of the groups $\mathbf{G}(E)$ or $\mathbf{G}^\omega(E)$. Let $\mathbf{P}$ and $\mathbf{B}$ be respectively a splitting parabolic subgroup and a splitting Borel subgroup of $\mathbf{G}$, both containing the  splitting Cartan subgroup $\mathbf{H}=\mathbf{H}(E)$ or $\mathbf{H}^\omega(E)$.
From the previous section we know that the homogeneous space $\mathbf{G}/\mathbf{P}$ can be viewed as an ind-variety of generalized flags of the form $\mathbf{Fl}(\mathcal{F},E)$ or $\mathbf{Fl}(\mathcal{F},\omega,E)$.
In this section we describe the decomposition of $\mathbf{G}/\mathbf{P}$ into $\mathbf{B}$-orbits. The main results are stated in Theorem \ref{theorem-1} in the case of $\mathbf{G}=\mathbf{G}(E)$ and in Theorem \ref{theorem-2} in the case of $\mathbf{G}=\mathbf{G}^\omega(E)$. In both cases it is shown that the $\mathbf{B}$-orbits form a cell decomposition of $\mathbf{G}/\mathbf{P}$, and their dimensions and closures are expressed in combinatorial terms.
In Section \ref{section-4-3} we derive the decomposition of the ind-group $\mathbf{G}$ into double cosets.
Unlike the case of Kac--Moody groups, the $\mathbf{B}$-orbits of $\mathbf{G}/\mathbf{P}$ can be infinite dimensional. The cases where all orbits are finite dimensional (resp., infinite dimensional) are characterized in Section \ref{section-4-4}.
In Section \ref{section-4-5} we focus on the situation where $\mathbf{G}/\mathbf{P}$ is an ind-grassmannian.

In this section the results are stated without proofs.\ The proofs are given in Section \ref{section-5}.

\subsection{Decomposition of $\mathbf{Fl}(\mathcal{F},E)$}

\label{section-4-1}

Let $\mathbf{G}=\mathbf{G}(E)$, $\mathbf{H}=\mathbf{H}(E)$, and $\mathbf{P}$, $\mathbf{B}$ be as above.
By Propositions \ref{proposition-1}--\ref{proposition-2} there is a generalized flag $\mathcal{F}$ compatible with $E$ such that $\mathbf{P}=\mathbf{P}_\mathcal{F}$ is the subgroup of elements $g\in\mathbf{G}(E)$ preserving $\mathcal{F}$, and the homogenenous space $\mathbf{G}(E)/\mathbf{P}$ is isomorphic to the ind-variety of generalized flags $\mathbf{Fl}(\mathcal{F},E)$.
The precise description of the decomposition of $\mathbf{Fl}(\mathcal{F},E)$ into $\mathbf{B}$-orbits is the object of this section.\ It requires some preliminaries and notation.

We denote by $\mathbf{W}(E)$ the group of permutations
$w:E\to E$ such that $w(e)=e$ for all but finitely many $e\in E$.
In particular, $\mathbf{W}(E)$ is isomorphic to the infinite symmetric group $\mathfrak{S}_\infty$.
Note that we have
\[\mathbf{W}(E)=\bigcup_{n\geq 1}W(E_n)\]
where $W(E_n)$ is the Weyl group of $G(E_n)$.

Let $(A,\preceq_A):=(A_\mathcal{F},\subset)$ be the set of pairs of consecutive elements of $\mathcal{F}$, so that $\mathcal{F}\in\mathbf{Fl}_A(V)$ and in fact $\mathbf{Fl}(\mathcal{F},E)\subset\mathbf{Fl}_A(V)$. Let $\mathfrak{S}(E,A)$ be the set of surjective maps $\sigma:E\to A$. For $\sigma\in\mathfrak{S}(E,A)$,
we denote by $\mathcal{F}_\sigma$ the generalized flag $\mathcal{F}_\sigma=\{F'_{\sigma,\alpha},F''_{\sigma,\alpha}:\alpha\in A\}$ given by
\begin{equation}
\label{relation-8}
F'_{\sigma,\alpha}=\langle e\in E:\sigma(e)\prec_A \alpha\rangle\quad\mbox{and}\quad F''_{\sigma,\alpha}=\langle e\in E:\sigma(e)\preceq_A\alpha\rangle.
\end{equation}
Thus $\{\mathcal{F}_\sigma:\sigma\in\mathfrak{S}(E,A)\}$ are exactly the generalized flags of $\mathbf{Fl}_A(V)$ compatible with the basis $E$ (see Definition \ref{definition-3}).
Let $\sigma_0:E\to A$ be the surjective map such that $\mathcal{F}=\mathcal{F}_{\sigma_0}$.

\begin{remark}
\label{remark-4}
The totally ordered set $(A,\preceq_A)$ and the surjective map $\sigma_0:E\to A$ give rise to a partial order $\preceq_\mathbf{P}$ on $E$, defined by letting $e\prec_\mathbf{P}e'$ if $\sigma_0(e)\prec_A\sigma_0(e')$.
Note that the partial order $\preceq_\mathbf{P}$ has the property
\begin{equation}
\label{10-new}
\parbox{12cm}{the relation ``$e$ is not comparable with $e'$'' (i.e., neither $e\prec_\mathbf{P}e'$ nor $e'\prec_\mathbf{P}e$) is an equivalence relation.}
\end{equation}
In fact, fixing
a splitting parabolic subgroup $\mathbf{P}\subset\mathbf{G}(E)$ containing $\mathbf{H}(E)$ is equivalent to fixing a partial order $\preceq_\mathbf{P}$ on $E$ satisfying property (\ref{10-new}). Moreover, $\mathbf{P}$ is a splitting Borel subgroup if and only if the order $\preceq_\mathbf{P}$ is total.
\end{remark}

The group $\mathbf{W}(E)$ acts on the set $\mathfrak{S}(E,A)$, hence on $E$-compatible generalized flags of $\mathbf{Fl}_A(V)$, by the procedure $(w,\sigma)\mapsto \sigma\circ w^{-1}$.
Let $\mathbf{W}_\mathbf{P}(E)\subset\mathbf{W}(E)$ be the subgroup of permutations such that $\sigma_0\circ w^{-1}=\sigma_0$.
Equivalently,
$\mathbf{W}_\mathbf{P}(E)$ is the subgroup of permutations $w\in\mathbf{W}(E)$
which preserve the fibers $\sigma_0^{-1}(\alpha)$ ($\alpha\in A$) of the map $\sigma_0$.

\begin{lemma}
\label{lemma-2-new}
The map $w\mapsto \mathcal{F}_{\sigma_0\circ w^{-1}}$ induces a bijection between the quotient $\mathbf{W}(E)/\mathbf{W}_\mathbf{P}(E)$ and the set of $E$-compatible generalized flags of the ind-variety $\mathbf{Fl}(\mathcal{F},E)$.
\end{lemma}

Let $\mathbf{W}(E)\cdot\sigma_0=\{\sigma_0\circ w^{-1}:w\in\mathbf{W}(E)\}$ denote the $\mathbf{W}(E)$-orbit of $\sigma_0$.

The splitting Borel subgroup $\mathbf{B}$ is the subgroup $\mathbf{B}=\mathbf{P}_{\mathcal{F}_0}$ of elements $g\in\mathbf{G}(E)$ preserving a maximal generalized flag $\mathcal{F}_0$ compatible with $E$ (see Proposition \ref{proposition-1}).
Equivalently $\mathbf{B}$ corresponds to a total order $\preceq_\mathbf{B}$ on the basis $E$ (see Remark \ref{remark-4}). Then, the generalized flag $\mathcal{F}_0=\{F'_{0,e},F''_{0,e}:e\in E\}$ is given by
$F'_{0,e}=\langle e'\in E:e'\prec_\mathbf{B}e\rangle$ and
$F''_{0,e}=\langle e'\in E:e'\preceq_\mathbf{B}e\rangle$ for all $e\in E$.

Relying on the total order $\preceq_\mathbf{B}$, we define a notion of inversion number and an analogue of the Bruhat order on the set $\mathfrak{S}(E,A)$.

{\it Number of inversions $n_\mathrm{inv}(\sigma)$.}
We say that a pair $(e,e')\in E\times E$ is an {\em inversion} of $\sigma\in\mathfrak{S}(E,A)$
if $e\prec_\mathbf{B}e'$ and $\sigma(e)\succ_A\sigma(e')$.
Then
\[n_\mathrm{inv}(\sigma):=|\{(e,e')\in E\times E:\mbox{$(e,e')$ is an inversion of $\sigma$}\}|\]
is the {\em inversion number of $\sigma$}.

\begin{remark}
{\rm (a)} The inversion number $n_\mathrm{inv}(\sigma)$ may be infinite. \\
{\rm (b)} If $\sigma\in\mathbf{W}(E)\cdot\sigma_0$, say $\sigma=\sigma_0\circ w$ with $w\in\mathbf{W}(E)$, then the inversion number of $\sigma$ is also given by the formula
\[n_\mathrm{inv}(\sigma)=|\{(e,e')\in E\times E:e\prec_\mathbf{B}e'\mbox{ and }w(e)\succ_\mathbf{P}w(e')\}|\]
(see Remark \ref{remark-4}). Note that the inversion number $n_\mathrm{inv}(\sigma)$ cannot be directly interpreted as a Bruhat length because we do not assume  $\mathbf{B}$ to be conjugate to a subgroup of $\mathbf{P}$.
\end{remark}

{\it Partial order $\leq$ on $\mathfrak{S}(E,A)$.}
We now define a partial order on the set $\mathfrak{S}(E,A)$, analogous to the Bruhat order.
For $(e,e')\in E\times E$ with $e\not=e'$, we denote by $t_{e,e'}$ the element of $\mathbf{W}(E)$ which exchanges $e$ with $e'$ and fixes every other element $e''\in E$.
Let $\sigma,\tau\in\mathfrak{S}(E,A)$.
We set $\sigma\hat{<}\tau$ if $\tau=\sigma\circ t_{e,e'}$ for a pair $(e,e')\in E\times E$ satisfying
$e\prec_\mathbf{B}e'$ and $\sigma(e)\prec_A\sigma(e')$.
We set $\sigma<\tau$ if there is a chain $\tau_0=\sigma\hat<\tau_1\hat<\tau_2\hat<\ldots\hat<\tau_k=\tau$ of elements of $\mathfrak{S}(E,A)$ (with $k\geq 1$).

{\it Element $\sigma_\mathcal{G}\in\mathfrak{S}(E,A)$.}
Given a generalized flag $\mathcal{G}=\{G'_\alpha,G''_\alpha:\alpha\in A\}\in\mathbf{Fl}_A(V)$ weakly compatible with $E$, we define an element $\sigma_\mathcal{G}\in\mathfrak{S}(E,A)$ which measures the relative position of $\mathcal{G}$ to the maximal generalized flag $\mathcal{F}_0$.
Set
\begin{equation}
\label{definition-sigmaG}
\sigma_\mathcal{G}(e)=\min\{\alpha\in A:G''_\alpha\cap F''_{0,e}\not= G''_\alpha\cap F'_{0,e}\} \ \mbox{ for all $e\in E$}.
\end{equation}
[It can be checked directly that the so obtained map $\sigma_\mathcal{G}:E\to A$  is indeed surjective, hence an element of $\mathfrak{S}(E,A)$. This fact is also shown in Section \ref{section-5-2} in the proof of Theorem \ref{theorem-2}.]

We are now in position to formulate the statement which describes the decomposition of $\mathbf{Fl}(\mathcal{F},E)$ into $\mathbf{B}$-orbits.

\begin{theorem}
\label{theorem-1}
Let $\mathbf{P}_\mathcal{F}$ be the splitting parabolic subgroup of $\mathbf{G}(E)$ containing $\mathbf{H}(E)$, and corresponding to a generalized flag $\mathcal{F}=\mathcal{F}_{\sigma_0}\in\mathbf{Fl}_A(V)$ (with $\sigma_0\in\mathfrak{S}(E,A)$) compatible with $E$.\ Let $\mathbf{B}$ be any splitting Borel subgroup of $\mathbf{G}(E)$ containing $\mathbf{H}(E)$. \\
{\rm (a)} We have the decomposition
\[\mathbf{G}(E)/\mathbf{P}_\mathcal{F}=\mathbf{Fl}(\mathcal{F},E)=\bigsqcup_{\sigma\in\mathbf{W}(E)\cdot\sigma_0}\mathbf{B}\mathcal{F}_\sigma=\bigsqcup_{w\in \mathbf{W}(E)/\mathbf{W}_\mathbf{P}(E)}\mathbf{B}\mathcal{F}_{\sigma_0\circ w^{-1}}.\]
{\rm (b)}
A generalized flag $\mathcal{G}\in\mathbf{Fl}(\mathcal{F},E)$ belongs to the $\mathbf{B}$-orbit $\mathbf{B}\mathcal{F}_\sigma$ ($\sigma\in\mathbf{W}(E)\cdot\sigma_0$) if and only if $\sigma_\mathcal{G}=\sigma$. \\
{\rm (c)} The orbit $\mathbf{B}\mathcal{F}_\sigma$ ($\sigma\in\mathbf{W}(E)\cdot\sigma_0$) is a locally closed ind-subvariety of $\mathbf{Fl}(\mathcal{F},E)$ isomorphic to the affine space $\mathbb{A}^{n_\mathrm{inv}(\sigma)}$ (which is infinite dimensional if $n_\mathrm{inv}(\sigma)$ is infinite). \\
{\rm (d)} For $\sigma,\tau\in\mathbf{W}(E)\cdot\sigma_0$, the inclusion $\mathbf{B}\mathcal{F}_\sigma\subset\overline{\mathbf{B}\mathcal{F}_\tau}$ holds if and only if $\sigma\leq\tau$.
\end{theorem}

\subsection{Decomposition of $\mathbf{Fl}(\mathcal{F},\omega,E)$}

\label{section-4-2}

In this section the basis  $E$ is  $\omega$-isotropic with corresponding involution $i_E:E\to E$ (see Section \ref{section-3-2}).
Let $\mathbf{P}\subset\mathbf{G}^\omega(E)$ be a splitting parabolic subgroup containing $\mathbf{H}^\omega(E)$,
or equivalently let $\mathcal{F}$ be an $\omega$-isotropic generalized flag compatible with $E$ (see Proposition \ref{proposition-3-new}).
Let $\mathbf{B}\subset\mathbf{G}^\omega(E)$ be a splitting Borel subgroup containing $\mathbf{H}^\omega(E)$.
We study the decomposition of the ind-variety $\mathbf{G}^\omega(E)/\mathbf{P}\cong \mathbf{Fl}(\mathcal{F},\omega,E)$ into $\mathbf{B}$-orbits.

Let $(A,\preceq_A,i_A)$ be a totally ordered set with involutive anti-automorphism
$i_A$, such that $\mathcal{F}\in \mathbf{Fl}_A^\omega(V)$. We denote by $\mathfrak{S}^\omega(E,A)$ the set of surjective maps $\sigma:E\to A$ such that $\sigma(i_E(e))=i_A(\sigma(e))$ for all $e\in E$. By Lemma \ref{lemma-1-new}, $\{\mathcal{F}_\sigma:\sigma\in\mathfrak{S}^\omega(E,A)\}$ are exactly the elements of $\mathbf{Fl}_A^\omega(V)$ compatible with $E$
(the notation $\mathcal{F}_\sigma$ is introduced in (\ref{relation-8})).
Let $\sigma_0\in\mathfrak{S}^\omega(E,A)$ be such that $\mathcal{F}=\mathcal{F}_{\sigma_0}$.

The group $\mathbf{W}^\omega(E)$ is defined as the group of permutations $w:E\to E$ such that $w(e)=e$ for all but finitely many $e\in E$ and $w(i_E(e))=i_E(w(e))$ for all $e\in E$. Note that  $\mathbf{W}^\omega(E)$ acts on the set $\mathfrak{S}^\omega(E,A)$ by the procedure $(w,\sigma)\mapsto \sigma\circ w^{-1}$. Let $\mathbf{W}_\mathbf{P}^\omega(E)$ be the subgroup of elements $w\in\mathbf{W}^\omega(E)$ such that $\sigma_0\circ w^{-1}=\sigma_0$ and let $\mathbf{W}^\omega(E)\cdot \sigma_0:=\{\sigma_0\circ w^{-1}:w\in\mathbf{W}^\omega(E)\}$ be the $\mathbf{W}^\omega(E)$-orbit of $\sigma_0$.

\begin{lemma}
\label{lemma-3}
The map $w\mapsto \mathcal{F}_{\sigma_0\circ w^{-1}}$ induces a bijection between $\mathbf{W}^\omega(E)/\mathbf{W}_\mathbf{P}^\omega(E)$ and the set of $E$-compatible elements of $\mathbf{Fl}(\mathcal{F},\omega,E)$.
\end{lemma}

The splitting Borel subgroup $\mathbf{B}$ is the subgroup $\mathbf{B}=\mathbf{P}_{\mathcal{F}_0}^\omega$ of elements preserving some maximal $\omega$-isotropic generalized flag $\mathcal{F}_0$ compatible with $E$. We can write $\mathcal{F}_0=\{F'_{0,e},F''_{0,e}:e\in E\}$ with $F'_{0,e}=\langle e'\in E:e'\prec_\mathbf{B}e\rangle$
and $F''_{0,e}=\langle e'\in E:e'\preceq_\mathbf{B}e\rangle$, where $\preceq_\mathbf{B}$ is a total order on $E$. Moreover, the fact that $\mathcal{F}_0$ is $\omega$-isotropic implies that the involution $i_E:E\to E$ is an anti-automorphism of the ordered set $(E,\preceq_\mathbf{B})$.

{\it Number of inversions $n_\mathrm{inv}^\omega(\sigma)$.} Let $\sigma\in\mathfrak{S}^\omega(E,A)$. We define an {\em $\omega$-isotropic inversion} of $\sigma$ as a pair $(e,e')\in E\times E$ such that
\[e\prec_\mathbf{B}e',\quad e\prec_\mathbf{B}i_E(e),\quad e'\not=i_E(e'),\ \mbox{ and }\ \sigma(e)\succ_A\sigma(e').\]
Let
\[n_\mathrm{inv}^\omega(\sigma)=|\{(e,e')\in E\times E:\mbox{$(e,e')$ is an $\omega$-isotropic inversion of $\sigma$}\}|.\]

{\em Partial order $\leq_\omega$ on $\mathfrak{S}^\omega(E,A)$.}
Given $(e,e')\in E\times E$ with $e\not=e'$, $i_E(e)\not=e$, $i_E(e')\not=e'$, we set
\[t_{e,e'}^\omega=t_{e,e'}\circ t_{i_E(e),i_E(e')}\mbox{ if }e'\not=i_E(e),\quad t_{e,e'}^\omega=t_{e,e'}\mbox{ if }e'=i_E(e).\]
Thus $t_{e,e'}^\omega\in\mathbf{W}^\omega(E)$. Let $\sigma,\tau\in\mathfrak{S}^\omega(E,A)$.
We set $\sigma\hat{<}_\omega\tau$ if $\tau=\sigma\circ t^\omega_{e,e'}$ for a pair $(e,e')$ satisfying
$e\prec_\mathbf{B}e'$ and $\sigma(e)\prec_A\sigma(e')$.
Finally we set $\sigma<_\omega\tau$ if there is a chain $\tau_0=\sigma\hat<_\omega\tau_1\hat<_\omega\tau_2\hat<_\omega\ldots\hat<_\omega\tau_k=\tau$ of elements of $\mathfrak{S}^\omega(E,A)$.

\begin{theorem}
\label{theorem-2}
Let $\mathbf{P}_\mathcal{F}^\omega$ be the splitting parabolic subgroup of $\mathbf{G}^\omega(E)$ containing $\mathbf{H}^\omega(E)$, and corresponding to an $E$-compatible generalized flag $\mathcal{F}=\mathcal{F}_{\sigma_0}\in\mathbf{Fl}^\omega_A(V)$ (with $\sigma_0\in\mathfrak{S}^\omega(E,A)$).\ Let $\mathbf{B}$ be any splitting Borel subgroup of $\mathbf{G}^\omega(E)$ containing $\mathbf{H}^\omega(E)$. \\
{\rm (a)} We have the decomposition
\[\mathbf{G}^\omega(E)/\mathbf{P}_\mathcal{F}^\omega=\mathbf{Fl}(\mathcal{F},\omega,E)=\bigsqcup_{\sigma\in\mathbf{W}^\omega(E)\cdot\sigma_0}\mathbf{B}\mathcal{F}_\sigma=\bigsqcup_{w\in \mathbf{W}^\omega(E)/\mathbf{W}^\omega_\mathbf{P}(E)}\mathbf{B}\mathcal{F}_{\sigma_0\circ w^{-1}}.\]
{\rm (b)}
For $\mathcal{G}\in\mathbf{Fl}(\mathcal{F},\omega,E)$ the map $\sigma_\mathcal{G}:E\to A$ (see (\ref{definition-sigmaG}))
belongs to $\mathbf{W}^\omega(E)\cdot\sigma_0$.
Moreover, $\mathcal{G}$ belongs to $\mathbf{B}\mathcal{F}_\sigma$ ($\sigma\in\mathbf{W}^\omega(E)\cdot\sigma_0$) if and only if $\sigma_\mathcal{G}=\sigma$. \\
{\rm (c)} The orbit $\mathbf{B}\mathcal{F}_\sigma$ ($\sigma\in\mathbf{W}^\omega(E)\cdot\sigma_0$) is a locally closed ind-subvariety of $\mathbf{Fl}(\mathcal{F},\omega,E)$ isomorphic to the affine space $\mathbb{A}^{n^\omega_\mathrm{inv}(\sigma)}$ (possibly infinite-dimensional). \\
{\rm (d)} For $\sigma,\tau\in\mathbf{W}^\omega(E)\cdot\sigma_0$, the inclusion $\mathbf{B}\mathcal{F}_\sigma\subset\overline{\mathbf{B}\mathcal{F}_\tau}$ holds if and only if $\sigma\leq_\omega\tau$.
\end{theorem}

\subsection{Bruhat decomposition of the ind-group $\mathbf{G}=\mathbf{G}(E)$ or $\mathbf{G}^\omega(E)$}

\label{section-4-3}

Let $\mathbf{H}=\mathbf{H}(E)$ or $\mathbf{H}^\omega(E)$, and let $\mathbf{W}=\mathbf{W}(E)$ or $\mathbf{W}^\omega(E)$.
If $\mathbf{W}=\mathbf{W}(E)$, the linear extension of $w\in\mathbf{W}$ is an element $\hat{w}\in\mathbf{G}(E)$. If $\mathbf{W}=\mathbf{W}^\omega(E)$, we can find scalars $\lambda_e\in\mathbb{K}^*$ ($e\in E$) such that the map $e\mapsto \lambda_e w(e)$ linearly extends to an element $\hat{w}\in\mathbf{G}^\omega(E)$. In both situations it is easy to deduce that $\mathbf{W}$ is isomorphic to the quotient $N_\mathbf{G}(\mathbf{H})/\mathbf{H}$.

Given a splitting parabolic subgroup $\mathbf{P}\subset\mathbf{G}$ containing $\mathbf{H}$, we denote by $\mathbf{W}_\mathbf{P}$ the corresponding subgroup of $\mathbf{W}$. The following statement describes the decomposition of the ind-group $\mathbf{G}$ into double cosets. It is a consequence of Theorems \ref{theorem-1}--\ref{theorem-2}.

\begin{corollary}
\label{corollary-1-Bruhat}
Let $\mathbf{G}$ be one of the ind-groups $\mathbf{G}(E)$ or $\mathbf{G}^\omega(E)$, and let $\mathbf{P}$ and $\mathbf{B}$ be respectively a splitting parabolic and a splitting subgroup of $\mathbf{G}$ containing $\mathbf{H}$. Then we have a decomposition
\[\mathbf{G}=\bigsqcup_{w\in\mathbf{W}/\mathbf{W}_\mathbf{P}}\mathbf{B}\hat{w}\mathbf{P}.\]
\end{corollary}

\begin{remark}
{\rm (a)}
Note that the unique assumption on the splitting parabolic and Borel subgroups $\mathbf{P}$ and $\mathbf{B}$ in Corollary \ref{corollary-1-Bruhat} is that they contain a common splitting Cartan subgroup, in particular it is not required that $\mathbf{B}$ be conjugate to a subgroup of $\mathbf{P}$. \\
{\rm (b)}
The ind-group $\mathbf{G}$ admits a natural exhaustion $\mathbf{G}=\bigcup_{n\geq 1}G_n$ by finite-dimensional subgroups of the form $G_n=G(E_n)$ or $G_n=G^\omega(E_n)$ (see Section \ref{section-2-2}). Moreover, the intersections $P_n:=\mathbf{P}\cap G_n$ and $B_n:=\mathbf{B}\cap G_n$ are respectively a parabolic subgroup and a Borel subgroup of $G_n$, containing a common Cartan subgroup. Then the decomposition of Corollary \ref{corollary-1-Bruhat} can be retrieved by considering usual Bruhat decompositions of the groups $G_n$ into double cosets for $P_n$ and $B_n$.
\end{remark}

\subsection{On the existence of cells of finite or infinite dimension}

\label{section-4-4}

In Theorems \ref{theorem-1}--\ref{theorem-2} it appears that the decomposition of an ind-variety of generalized flags into $\mathbf{B}$-orbits may comprise orbits of infinite dimension.
The following result determines precisely the situations in which infinite-dimensional orbits arise.

\begin{theorem}
\label{theorem-3}
Let $\mathbf{G}$ be one of the groups $\mathbf{G}(E)$ or $\mathbf{G}^\omega(E)$. Let $\mathbf{P},\mathbf{B}\subset\mathbf{G}$ be splitting parabolic and Borel subgroups containing the splitting Cartan subgroup $\mathbf{H}$ of $\mathbf{G}$.  \\
{\rm (a)} The following conditions are equivalent:
\begin{itemize}
\item[\rm (i)] $\mathbf{B}$ is conjugate (under $\mathbf{G}$) to a subgroup of $\mathbf{P}$;
\item[\rm (ii)] At least one $\mathbf{B}$-orbit of $\mathbf{G}/\mathbf{P}$ is finite dimensional; \item[\rm (iii)] One $\mathbf{B}$-orbit of $\mathbf{G}/\mathbf{P}$ is a single point (and this orbit is necessarily unique).
\end{itemize}
{\rm (b)} Let $\preceq_\mathbf{B}$ be the total order on the basis $E$ induced by $\mathbf{B}$. Assume that $\mathbf{P}\not=\mathbf{G}$. The following conditions are equivalent:
\begin{itemize}
\item[\rm (i)] $\mathbf{B}$ is conjugate (under $\mathbf{G}$) to a subgroup of $\mathbf{P}$, and the ordered set $(E,\preceq_\mathbf{B})$ is isomorphic (as ordered set) to a subset of $(\mathbb{Z},\leq)$;
\item[\rm (ii)] Every $\mathbf{B}$-orbit of $\mathbf{G}/\mathbf{P}$ is finite dimensional.
\end{itemize}
\end{theorem}

\begin{remark}
{\rm (a)}
Theorem \ref{theorem-3} provides in particular a criterion for a given splitting Borel subgroup to be conjugate to a subgroup of a given splitting parabolic subgroup. This criterion is applied in the next section. \\
{\rm (b)}
Following \cite{DP}, we call a generalized flag $\mathcal{G}$ a {\em flag} if the chain $(\mathcal{G},\subset)$ is isomorphic as ordered set to a subset of $(\mathbb{Z},\leq)$.
Then the second part of condition {\rm (b)}\,{\rm (i)} in Theorem \ref{theorem-3} can be rephrased by saying that the maximal generalized flag $\mathcal{F}_0$ is a flag. Another characterization of flags is provided by \cite[Proposition 7.2]{DP} which says that the ind-variety of generalized flags $\mathbf{Fl}(\mathcal{G},E)$ (resp., $\mathbf{Fl}(\mathcal{G},\omega,E)$) is {\em projective} (i.e., isomorphic as ind-variety to a closed ind-subvariety of the infinite-dimensional projective space $\mathbb{P}^\infty$) if and only if $\mathcal{G}$ is a flag.
\end{remark}

\subsection{Decomposition of ind-grassmannians}

\label{section-4-5}

A minimal (nontrivial) generalized flag $\mathcal{F}=\{0,F,V\}$ of the space $V$ is determined by the proper nonzero subspace $F\subset V$. If $\mathcal{F}$ is compatible with the basis $E$, then the surjective map $\sigma_0:E\to \{1,2\}$ such that $F=\langle e\in E:\sigma_0(e)=1\rangle$ can be simply viewed as the subset $\sigma_0\subset E$ such that $F=\langle \sigma_0\rangle$.

In this case the ind-variety $\mathbf{Fl}(\mathcal{F},E)$ is an {\em ind-grassmannian} and we denote it by $\mathbf{Gr}(F,E)$.
\begin{itemize}
\item
If $k:=\dim F$ is finite, a subspace $F_1\subset V$ is $E$-commensurable with $F$ if and only if $\dim F_1=k$.
Thus the ind-variety $\mathbf{Gr}(F,E)$ only depends on $k$, and we write $\mathbf{Gr}(k)=\mathbf{Gr}(F,E)$ in this case.
\item
If $k:=\mathrm{codim}_VF$ is finite, the ind-variety $\mathbf{Gr}(F,E)$ depends on $E$ and $k$ (but not on $F$). It is also isomorphic to $\mathbf{Gr}(k)$. Indeed, the basis $E\subset V$ gives rise to a dual family $E^*\subset V^*$. The linear space $V_*:=\langle E^*\rangle$ is then countable dimensional. Let $U^\#:=\{\phi\in V_*:\phi(u)=0\ \forall u\in U\}$ be the orthogonal subspace in $V_*$ of a subspace $U\subset V$. The map $U\mapsto U^\#$ realizes an isomorphism of ind-varieties between $\mathbf{Gr}(F,E)$ and $\{F'\subset V_*:\dim F'=k\}\cong\mathbf{Gr}(k)$.
\item
If $F$ is both infinite dimensional and infinite codimensional, the ind-variety $\mathbf{Gr}(F,E)$ depends on $(F,E)$, although all ind-varieties of this type are isomorphic; their isomorphism class is denoted $\mathbf{Gr}(\infty)$.
Moreover, $\mathbf{Gr}(\infty)$ and $\mathbf{Gr}(k)$ are not isomorphic as ind-varieties (see \cite{PT}).
\end{itemize}

Let $\mathfrak{S}(E)$ be the set of subsets $\sigma\subset E$. The  group $\mathbf{W}(E)$ acts on $\mathfrak{S}(E)$ in a natural way. The $\mathbf{W}(E)$-orbit of $\sigma_0$ is the subset $\mathbf{W}(E)\cdot\sigma_0=\{\sigma\in\mathfrak{S}(E):|\sigma_0\setminus\sigma|=|\sigma\setminus\sigma_0|<+\infty\}$.
We write $F_\sigma=\langle\sigma\rangle$ (for $\sigma\in\mathfrak{S}(E)$).

The following statement describes the decomposition of the ind-grassmannian $\mathbf{Gr}(F,E)$ into $\mathbf{B}$-orbits. It is a direct consequence of Theorem \ref{theorem-1}.

\begin{proposition}
\label{proposition-4-new}
Let $\mathbf{B}\subset\mathbf{G}(E)$ be a splitting Borel subgroup containing $\mathbf{H}(E)$. \\
{\rm (a)}
We have the decomposition
\[\mathbf{Gr}(F,E)=\bigsqcup_{\sigma\in\mathbf{W}(E)\cdot\sigma_0}\mathbf{B}F_\sigma.\]
{\rm (b)} For $F'\in\mathbf{Gr}(F,E)$,\ we have $F'\in\mathbf{B}F_\sigma$ if and only if \[\sigma=\sigma_{F'}:=\{e\in E:F'\cap\langle e'\in E:e'\prec_\mathbf{B}e\rangle\not= F'\cap\langle e'\in E:e'\preceq_\mathbf{B}e\rangle\}.\]
{\rm (c)} For $\sigma\in\mathbf{W}(E)\cdot\sigma_0$, the orbit $\mathbf{B}F_\sigma$ is a locally closed ind-subvariety of $\mathbf{Gr}(F,E)$ isomorphic to an affine space $\mathbb{A}^{d_\sigma}$ of (possibly infinite) dimension
\[d_\sigma=n_\mathrm{inv}(\sigma):=|\{(e,e')\in E\times E:e\prec_\mathbf{B}e',\ e\notin\sigma,\ e'\in\sigma\}|.\]
{\rm (d)} For $\sigma,\tau\in\mathbf{W}(E)\cdot\sigma_0$, the inclusion $\mathbf{B}F_\sigma\subset\overline{\mathbf{B}F_\tau}$ holds if and only if $\sigma\leq\tau$, where the relation $\sigma\leq\tau$ means that,
if $e_1\prec_\mathbf{B}e_2\prec_\mathbf{B}\ldots\prec_\mathbf{B}e_\ell$ are the elements of $\sigma\setminus\tau$ and $f_1\prec_\mathbf{B}f_2\prec_\mathbf{B}\ldots\prec_\mathbf{B}f_\ell$ are the elements of $\tau\setminus\sigma$, then
$e_i\prec_\mathbf{B}f_i$ for all $i\in\{1,\ldots,\ell\}$.
\end{proposition}

\begin{example}[Case of the ind-grassmannian $\mathbf{Gr}(k)$]
\label{example-2}
Let $\mathfrak{S}_k(E)$ be the set of subsets $\sigma\subset E$ of cardinality $k$.
Given $\sigma_0\in\mathfrak{S}_k(E)$,
set $F=\langle \sigma_0\rangle$, and consider the splitting parabolic subgroup $\mathbf{P}_F=\{g\in\mathbf{G}(E):g(F)=F\}$
and the ind-grassmannian $\mathbf{Gr}(k)=\mathbf{Gr}(F,E)=\mathbf{G}(E)/\mathbf{P}_F$.
 By Proposition \ref{proposition-4-new}\,{\rm (a)}, we have the decomposition
\[\mathbf{Gr}(k)=\bigsqcup_{\sigma\in\mathfrak{S}_k(E)}\mathbf{B}F_\sigma.\]
By Proposition \ref{proposition-4-new}\,{\rm (c)}, the cell $\mathbf{B}F_\sigma$ is finite dimensional if and only if $\sigma$ is contained in a finite ideal of the ordered set $(E,\preceq_\mathbf{B})$, i.e., there is a finite subset $\overline{\sigma}\subset E$ satisfying ($e\in \overline{\sigma}$ and $e'\preceq_\mathbf{B}e$ $\Rightarrow$ $e'\in\overline{\sigma}$) and containing $\sigma$.
It easily follows that there are finite-dimensional $\mathbf{B}$-orbits in $\mathbf{Gr}(k)$ if and only if  the maximal generalized flag $\mathcal{F}_0$ corresponding to $\mathbf{B}$ contains a subspace $M$ of dimension $k$.
By Theorem \ref{theorem-3}, $\mathbf{B}$ is conjugate to a subgroup of the splitting parabolic subgroup
$\mathbf{P}_F$
exactly in this case.
By Theorem \ref{theorem-3} (or directly), we note that
all cells $\mathbf{B}F_\sigma\subset\mathbf{Gr}(k)$ are finite dimensional if and only if $(E,\preceq_\mathbf{B})$ is isomorphic to $(\mathbb{N},\leq)$ as an ordered set, in other words $\mathcal{F}_0$ is a flag of the form
\begin{equation}
\label{10}
\mathcal{F}_0=(F_{0,0}\subset F_{0,1}\subset F_{0,2}\subset\ldots)\quad\mbox{with}\quad\dim F_{0,i}=i\ \mbox{ for all $i\geq 0$.}
\end{equation}
By Proposition \ref{proposition-4-new}\,{\rm (d)}, given $\sigma=\{e_1\prec_\mathbf{B}e_2\prec_\mathbf{B}\ldots\prec_\mathbf{B}e_k\}$
and $\tau=\{f_1\prec_\mathbf{B}f_2\prec_\mathbf{B}\ldots\prec_\mathbf{B}f_k\}$, we have $\mathbf{B}F_\sigma\subset\overline{\mathbf{B}F_\tau}$ if and only if $e_i\preceq_\mathbf{B}f_i$ for all $i\in\{1,\ldots,k\}$.

Now let $\tau_0\subset E$ be an infinite subset whose complement $E\setminus \tau_0$ is finite of cardinality $k$. Let $M=\langle\tau_0\rangle$ be the corresponding subspace of $V$ of codimension $k$ and let $\mathbf{P}_M\subset\mathbf{G}(E)$ be the corresponding splitting parabolic subgroup. We consider the ind-grassmannian $\mathbf{Gr}(M,E)=\mathbf{G}(E)/\mathbf{P}_M$ which is isomorphic to $\mathbf{Gr}(k)=\mathbf{G}(E)/\mathbf{P}_F$ as mentioned at the beginning of Section \ref{section-4-5}. If $\mathcal{F}_0$ is as in (\ref{10}) and $\mathbf{B}$ is the corresponding splitting Borel subgroup, then it follows from Proposition \ref{proposition-4-new}\,{\rm (c)} that every $\mathbf{B}$-orbit of $\mathbf{Gr}(M,E)$ is infinite dimensional. By Theorem \ref{theorem-3}, this shows in particular that the splitting parabolic subgroups $\mathbf{P}_F$ and $\mathbf{P}_M$ are not conjugate under $\mathbf{G}(E)$.
\end{example}

\begin{example}[Case of the infinite-dimensional projective space]
\label{example-3}
Assume that $k=\dim F=1$. In this case $\mathbf{Gr}(k)$ is the infinite-dimensional projective space $\mathbb{P}^\infty$ (see Example \ref{example1}). The decomposition becomes
\[\mathbb{P}^\infty=\bigsqcup_{e\in E}\mathbf{C}_e\]
where $\mathbf{C}_e=\mathbf{B}\langle e\rangle=\{L\mbox{ line}:L\subset \langle e'\in E:e'\preceq_\mathbf{B}e\rangle,\ L\not\subset\langle e'\in E:e'\prec_\mathbf{B} e\rangle\}$ for all $e\in E$. The cell $\mathbf{C}_e$ is isomorphic to an affine space of dimension $\dim \mathbf{C}_e=|\{e'\in E:e'\prec_{\mathbf{B}}e\}|$. Moreover, $\mathbf{C}_e\subset\overline{\mathbf{C}_f}$ if and only if $e\preceq_{\mathbf{B}} f$.

In this case the maximal generalized flag $\mathcal{F}_0=\{F'_{0,e},F''_{0,e}:e\in E\}$ corresponding to $\mathbf{B}$ can be retrieved from the cell decomposition:
\[F''_{0,e}=\sum_{L\in\overline{\mathbf{C}}_e}L\quad\mbox{and}\quad F'_{0,e}=\sum_{L\in\overline{\mathbf{C}}_e\setminus \mathbf{C}_e}L\ \mbox{ for all $e\in E$.}\]

More generally, let $(A,\preceq)$ be a totally ordered set and let $\mathbb{P}^\infty=\bigsqcup_{\alpha\in A}\mathbf{C}_\alpha$ be a {\it linear} cell decomposition such that $\mathbf{C}_\alpha\subset\overline{\mathbf{C}_\beta}$ whenever $\alpha\preceq\beta$. By ``linear'' we mean that each $\overline{\mathbf{C}_\alpha}$ is a projective subspace of $\mathbb{P}^\infty$, i.e., we can find a subspace $F''_\alpha\subset V$ such that $\overline{\mathbf{C}_\alpha}=\mathbb{P}(F''_\alpha)$.
Setting $F'_\alpha=\sum_{\beta<\alpha} F''_\beta$,
we get a generalized flag $\mathcal{F}_0:=\{F'_\alpha, F''_\alpha:\alpha\in A\}$ such that
$\mathbb{P}(F''_\alpha)\setminus \mathbb{P}(F'_\alpha)$ is a (possibly infinite-dimensional) affine space for all $\alpha$.
The last property ensures that $\dim F''_\alpha/F'_\alpha=1$, i.e., $\mathcal{F}_0$ is a maximal generalized flag. In this way we obtain a correspondence between maximal generalized flags (not necessarily compatible with a given basis) and linear cell decompositions of the infinite-dimensional projective space $\mathbb{P}^\infty$.
\end{example}

\begin{example}[Case of the ind-grassmannian $\mathbf{Gr}(\infty)$]

Assume that the basis $E$ is parametrized by $\mathbb{Z}$, in other words let $E=\{e_i\}_{i\in\mathbb{Z}}$. We consider the splitting Borel subgroup $\mathbf{B}$ corresponding to the natural order $\leq$ on $\mathbb{Z}$.

Let $F=\langle e_i:i\leq 0\rangle$. Then the ind-variety $\mathbf{Gr}(F,E)$ is isomorphic to $\mathbf{Gr}(\infty)$. We have $\mathbf{B}\subset\mathbf{P}_F$. It follows from Theorem \ref{theorem-3} that every $\mathbf{B}$-orbit of $\mathbf{Gr}(F,E)$ is finite dimensional.

Let $F'=\langle e_i:i\in 2\mathbb{Z}\rangle$. Again the ind-variety $\mathbf{Gr}(F',E)$ is isomorphic to $\mathbf{Gr}(\infty)$.\ However in this case we see from Proposition \ref{proposition-4-new}\,{\rm (c)} that every $\mathbf{B}$-orbit of $\mathbf{Gr}(F',E)$ is infinite dimensional.
\end{example}

We now suppose that the space $V$ is endowed with a nondegenerate symmetric or skew-symmetric bilinear form $\omega$ and the basis $E$ is $\omega$-isotropic with corresponding involution $i_E:E\to E$.
Then a minimal $\omega$-isotropic generalized flag is of the form $\mathcal{F}=(0\subset F\subset F^\perp\subset V)$ with $F\subset V$ proper and nontrivial, possibly $F=F^\perp$. Assuming that $F$ is compatible with the basis $E$, there is a subset $\sigma_0\subset E$ such that
$F=\langle \sigma_0\rangle$
and $i_E(\sigma_0)\cap\sigma_0=\emptyset$ as the generalized flag is $\omega$-isotropic.
The ind-variety $\mathbf{Fl}(\mathcal{F},\omega,E)$ is also denoted $\mathbf{Gr}(F,\omega,E)$ and called {\em isotropic ind-grassmannian}.
\begin{itemize}
\item If $\dim F=k$ is finite, the ind-variety $\mathbf{Gr}(F,\omega,E)$ is the set of all $k$-dimensional subspaces $M\subset V$ such that $M\subset M^\perp$. This ind-variety does not depend on $(F,E)$ and we denote it also by $\mathbf{Gr}^\omega(k)$.
\item If $\dim F$ is infinite, the isomorphism class of the ind-variety $\mathbf{Gr}(F,\omega,E)$ also depends on the dimension of the quotient $F^\perp/F$. A special situation is when $\dim F^\perp/F\in\{0,1\}$, in which case $\mathbf{Gr}(F,\omega,E)$ is formed by maximal isotropic subspaces.
\end{itemize}
We denote by $\mathfrak{S}^\omega(E)$ the set of subsets $\sigma\subset E$ such that $i_E(\sigma)\cap\sigma=\emptyset$. The group $\mathbf{W}^\omega(E)$  acts on $\mathfrak{S}^\omega(E)$ in a natural way. The orbit $\mathbf{W}^\omega(E)\cdot\sigma_0$ is the set of subsets $\sigma\in\mathfrak{S}^\omega(E)$ such that $|\sigma\setminus\sigma_0|=|\sigma_0\setminus\sigma|<+\infty$.
From Theorem \ref{theorem-2} we obtain the following description of the $\mathbf{B}$-orbits of $\mathbf{Gr}(F,\omega,E)$.

\begin{proposition}
Let $\mathbf{B}$ be a splitting Borel subgroup of $\mathbf{G}^\omega(E)$ corresponding to a total order $\preceq_\mathbf{B}$ on $E$. Recall that $i_E$ is a anti-automorphism of the ordered set $(E,\preceq_\mathbf{B})$. \\
{\rm (a)} We have the decomposition
\[\mathbf{Gr}(F,\omega,E)=\bigsqcup_{\sigma\in\mathbf{W}^\omega(E)\cdot\sigma_0}\mathbf{B}F_\sigma\]
where as before $F_\sigma=\langle\sigma\rangle$. \\
{\rm (b)} For $F'\in\mathbf{Gr}(F,\omega,E)$ we have $\sigma_{F'}\in\mathbf{W}^\omega(E)\cdot\sigma_0$ (see Proposition \ref{proposition-4-new}\,{\rm (b)}), moreover $F'\in\mathbf{B}F_\sigma$ if and only if $\sigma=\sigma_{F'}$. \\
{\rm (c)} For $\sigma\in\mathbf{W}^\omega(E)\cdot\sigma_0$, the orbit $\mathbf{B}F_\sigma$ is a locally closed ind-subvariety of $\mathbf{Gr}(F,\omega,E)$ isomorphic to an affine space of (possibly infinite) dimension
\[n_\mathrm{inv}^\omega(\sigma):=|\{(e,e')\in E\times E:e\prec_\mathbf{B}e'\not=i_E(e'),\ e\prec_\mathbf{B}i_E(e),\ \big((e\notin\sigma,\ e'\in\sigma)\ \mathrm{or}\ (i_E(e)\in\sigma,\ i_E(e')\notin\sigma)\big)\}|.\]
{\rm (d)} For  $\sigma,\tau\in\mathbf{W}^\omega(E)\cdot\sigma_0$,
the inclusion $\mathbf{B}F_\sigma\subset\overline{\mathbf{B}F_\tau}$ holds if and only if $\sigma\leq\tau$, where the relation $\sigma\leq\tau$ is defined as in Proposition \ref{proposition-4-new}\,{\rm (d)}.
\end{proposition}

\begin{example}[Case of the isotropic ind-grassmannian $\mathbf{Gr}^\omega(k)$]
In this case the cells $\mathbf{B}F_\sigma$ are parametrized by the set $\mathfrak{S}_k^\omega(E)$ of finite subsets $\sigma\subset E$ of cardinality $k$ such that $i_E(\sigma)\cap\sigma=\emptyset$.
The cell $\mathbf{B}F_\sigma$ is finite dimensional if and only if $\sigma$ is contained in a finite ideal $\overline{\sigma}$ of the ordered set $(E,\preceq_\mathbf{B})$. Thereby the ind-variety $\mathbf{Gr}^\omega(k)$ has finite-dimensional $\mathbf{B}$-orbits if and only if the ordered set $(E,\preceq_\mathbf{B})$ has a finite ideal with $k$ elements.
Equivalently, the maximal generalized flag $\mathcal{F}_0$ corresponding to $\mathbf{B}$ has a subspace $M\in\mathcal{F}_0$ of dimension $k$. Since $\mathcal{F}_0$ is maximal and $\omega$-isotropic, it is of the form
\[\mathcal{F}_0=\{0=F_{0,0}\subset F_{0,1}\subset\ldots\subset F_{0,k}\subset(\ldots)\subset F_{0,k}^\perp\subset\ldots\subset F_{0,1}^\perp\subset F_{0,0}^\perp=V\}\]
with infinitely many terms between $F_{0,k}$ and $F_{0,k}^\perp$.
Hence the ordered set $(\mathcal{F}_0,\subset)$ is not isomorphic to a subset of $(\mathbb{Z},\leq)$. By Theorem \ref{theorem-3}, this implies that $\mathbf{Gr}^\omega(k)$ admits infinite-dimensional $\mathbf{B}$-orbits. Therefore, contrary to the case of the  ind-grassmannian $\mathbf{Gr}(k)$ (see Example \ref{example-2}), there is no splitting Borel subgroup $\mathbf{B}\subset\mathbf{G}^\omega(E)$ for which all  $\mathbf{B}$-orbits of the  isotropic ind-grassmannian $\mathbf{Gr}^\omega(k)$ are finite dimensional.

Assume that $\omega$ is skew symmetric and $k=1$.\ Then $\mathbf{Gr}^\omega(k)$ coincides with the entire infinite-dimensional projective space $\mathbb{P}^\infty$. The above discussion shows that, for every splitting Borel subgroup $\mathbf{B}$ of $\mathbf{G}^\omega(E)$, there are  infinite-dimensional $\mathbf{B}$-orbits in the projective space $\mathbb{P}^\infty$. We know however from Examples \ref{example-2}--\ref{example-3} that, for a well-chosen splitting Borel subgroup of $\mathbf{G}(E)$, the infinite-dimensional projective space $\mathbb{P}^\infty$ admits a decomposition into finite-dimensional orbits. Therefore the realizations of $\mathbb{P}^\infty$ as $\mathbf{Gr}(1)$ and $\mathbf{Gr}^\omega(1)$ yield different sets of cell decompositions on $\mathbb{P}^\infty$.
\end{example}

\begin{example}[An isotropic ind-grassmannian with decomposition into finite-dimensional cells]
Let $E=\{e_i:i\in 2\mathbb{Z}+1\}$ be an $\omega$-isotropic basis of $V$ such that $\omega(e_i,e_j)=0$ unless $i+j=0$. For $k\geq 1$, we let $F=\langle e_i:i\leq -k\rangle$ and consider the ind-grassmannian $\mathbf{Gr}(F,\omega,E)$. Let $\mathbf{B}$ be the splitting Borel subgroup of $\mathbf{G}^\omega(E)$ corresponding to the natural total order $\leq$ on $2\mathbb{Z}+1$. We then have $\mathbf{B}\subset\mathbf{P}_F^\omega:=\{g\in\mathbf{G}^\omega(E):g(F)=F\}$, hence
by Theorem \ref{theorem-3}\,{\rm (b)} all the $\mathbf{B}$-orbits of the ind-grassmannian $\mathbf{Gr}(F,\omega,E)$ are finite dimensional.
\end{example}

\section{Proof of the results stated in Section \ref{section-4}}

\label{section-5}

Throughout this section let $\mathbf{G}=\mathbf{G}(E)$ or $\mathbf{G}^\omega(E)$, and $\mathbf{W}$ is the corresponding group $\mathbf{W}(E)$ or $\mathbf{W}^\omega(E)$ (see Sections \ref{section-4-1}--\ref{section-4-2}).
The proofs of the results stated in Section \ref{section-4} are given in Sections \ref{section-5-3}--\ref{section-5-5}. They rely on preliminary facts presented in Section \ref{section-5-1} (which is concerned with the combinatorics of the  group $\mathbf{W}$) and Section \ref{section-5-2} (where we review some standard facts on Schubert decomposition of finite-dimensional flag varieties).

\subsection{Combinatorial properties of the group $\mathbf{W}$}

\label{section-5-1}

We first recall certain features of the group $\mathbf{W}$:
\begin{itemize}
\item $\mathbf{W}\cong N_\mathbf{G}(\mathbf{H})/\mathbf{H}$ where $\mathbf{H}\subset\mathbf{G}$ is the splitting Cartan subgroup of elements diagonal in the basis $E$; specifically, to an element $w\in\mathbf{W}$, we can associate an explicit representative $\hat{w}\in N_\mathbf{G}(\mathbf{H})$ (see Section \ref{section-4-3}).
\item We have a natural exhaustion
\[\mathbf{W}=\bigcup_{n\geq 1}W_n\]
where $W_n=W(E_n)$ (resp. $W_n=W^\omega(E_n)$) is the Weyl group of $G_n=G(E_n)$ (resp. $G_n=G^\omega(E_n)$).
\item
Let $E'=E$ if $\mathbf{G}=\mathbf{G}(E)$ and $E'=\{e\in E:i_E(e)\not=e\}$ if $\mathbf{G}=\mathbf{G}^\omega(E)$, and let
\[\hat{E}=\{(e,e')\in E'\times E':e\not=e'\}.\]
For $(e,e')\in\hat{E}$, set $s_{e,e'}=t_{e,e'}$ if $\mathbf{G}=\mathbf{G}(E)$ and $s_{e,e'}=t^\omega_{e,e'}$ if $\mathbf{G}=\mathbf{G}^\omega(E)$ (see Sections \ref{section-4-1}--\ref{section-4-2}).
In both cases  for each pair $(e,e')\in\hat{E}$, we get an element $s_{e,e'}\in\mathbf{W}$. Clearly $\{s_{e,e'}:(e,e')\in\hat{E}\}$ is a system of generators of $\mathbf{W}$.
\end{itemize}

\subsubsection{Analogue of Bruhat length}

As seen in Sections \ref{section-4-1}--\ref{section-4-2}, fixing a splitting Borel subgroup $\mathbf{B}$ of $\mathbf{G}$ with $\mathbf{B}\supset\mathbf{H}$ is equivalent to fixing a total order $\preceq_\mathbf{B}$ on the basis $E$ (resp., such that the involution $i_E:E\to E$ becomes an anti-automorphism of ordered set, in the case where $\mathbf{G}=\mathbf{G}^\omega(E)$).
This total order allows us to define a system of simple transpositions for $\mathbf{W}$ by letting
\[S_\mathbf{B}=\{s_{e,e'}:\mbox{$e,e'$ are consecutive elements of $(E',\preceq_\mathbf{B})$}\}.\]
Note however that in general $S_\mathbf{B}$ does not generate the group $\mathbf{W}$. For  $w\in\mathbf{W}$, we define
\[\ell_\mathbf{B}(w)=\min\{m\geq 0:w=s_1\cdots s_m\ \mbox{for some $s_1,\ldots,s_m\in S_\mathbf{B}$}\}\]
if the set on the right-hand side is nonempty, and
\[\ell_\mathbf{B}(w)=+\infty\]
otherwise.

For every $n\geq 1$, the order $\preceq_\mathbf{B}$ induces a total order on the finite subset $E_n\subset E$, and thus a system of simple reflections $S_{\mathbf{B},n}:=\{s_{e,e'}:\mbox{$e,e'$ are consecutive elements of $(E_n\cap E',\preceq_\mathbf{B})$}\}$ of the Weyl group $W_n$. Let $\ell_{\mathbf{B},n}(w)$ be the usual Bruhat length of $w\in W_n$ with respect to $S_{\mathbf{B},n}$.




\begin{proposition}
\label{proposition-8-new}
Let $w\in\mathbf{W}$. Then
\begin{itemize}
\item[\rm (a)] $\ell_\mathbf{B}(w)=\lim_{n\to\infty}\ell_{\mathbf{B},n}(w)$; \item[\rm (b)] $\ell_\mathbf{B}(w)=\left\{
\begin{array}{ll}
|\{(e,e')\in \hat{E}:e\prec_\mathbf{B}e'\mbox{ and }w(e)\succ_\mathbf{B}w(e')\}| & \mbox{if $\mathbf{G}=\mathbf{G}(E)$,} \\
|\{(e,e')\in \hat{E}:e\prec_\mathbf{B}e',\ e\prec_\mathbf{B}i_E(e)\mbox{ and }w(e)\succ_\mathbf{B}w(e')\}| & \mbox{if $\mathbf{G}=\mathbf{G}^\omega(E)$;}
\end{array}\right.$
\item[\rm (c)] $\ell_\mathbf{B}(w)=+\infty$ if and only if there is $e\in E$ such that the set $\{e'\in E:e\prec_\mathbf{B}e'\prec_\mathbf{B}w(e)\}$ is infinite.
\end{itemize}
\end{proposition}

\begin{proof}
Denote by $m_\mathbf{B}(w)$ the quantity in the right-hand side of {\rm (b)}. Then
\begin{equation}
\label{5-1}
\ell_\mathbf{B}(w)\geq \lim_{n\to\infty}\ell_{\mathbf{B},n}(w)=m_\mathbf{B}(w)
\end{equation}
(the inequality is a consequence of the definitions of $\ell_\mathbf{B}(w)$ and $\ell_{\mathbf{B},n}(w)$ while the equality follows from properties of (finite) Weyl groups).

Let $I_e(w)=\{e'\in E:e\prec_\mathbf{B}e'\prec_\mathbf{B}w(e)\}$.
We claim that
\begin{equation}
\label{5-2}
m_\mathbf{B}(w)=+\infty\ \Leftrightarrow\ \exists e\in E\ \mbox{such that}\ |I_e(w)|=+\infty.
\end{equation}

We first check the implication $\Rightarrow$ in (\ref{5-2}). The assumption yields an infinite sequence $\{(e_i,e'_i)\}_{i\in\mathbb{N}}$ such that $e_i\prec_\mathbf{B}e'_i$ and $w(e_i)\succ_\mathbf{B}w(e'_i)$. Since $w$ fixes all but finitely many elements of $E$, one of the sequences $\{e_i\}_{i\in\mathbb{N}}$ and $\{e'_i\}_{i\in\mathbb{N}}$ has a stationary subsequence,  and thus along a relabeled subsequence $\{(e_i,e'_i)\}_{i\in\mathbb{N}}$ we have $e_i=e$ for all $i\in\mathbb{N}$ and some $e\in E$, or $e'_i=e'$ for all $i\in\mathbb{N}$ and some $e'\in E$. In the former case,
the set $\{e'_i:w(e'_i)=e'_i\}$ is infinite and contained in $I_e(w)$. In the latter case, we similarly obtain that the set $\{f\in E:w(e')\prec_\mathbf{B} f\prec_\mathbf{B} e'\}$ is infinite, and since $w$ has finite order, this implies that $I_{w^r(e')}(w)$ is infinite for some $r\geq 1$.

Next we check the implication $\Leftarrow$ in (\ref{5-2}).
We assume that $|I_e(w)|=+\infty$ for some $e\in E$.
(Then, necessarily, $w(e)\not=e$, hence $e\not=i_E(e)$ in the case where $\mathbf{G}=\mathbf{G}^\omega(E)$.)
Since $w$ fixes all but finitely many elements of $E$, the set
$\{e'\in I_e(w):w(e')=e'\}$ is infinite.
Therefore, there are infinitely many couples $(e,e')\in\hat{E}$ such that $e\prec_\mathbf{B}e'$ and $w(e)\succ_\mathbf{B}w(e')$. Moreover, in the case where $\mathbf{G}=\mathbf{G}^\omega(E)$, up to replacing $(e,e')$ by $(i_E(e'),i_E(e))$, we get infinitely many such couples satisfying $e\prec_\mathbf{B}i_E(e)$.
This implies $m_\mathbf{B}(w)=+\infty$, and (\ref{5-2}) is proved.

In view of (\ref{5-1}) and (\ref{5-2}), to complete the proof of the proposition, it remains to show the relation $\ell_\mathbf{B}(w)\leq m_\mathbf{B}(w)$. We argue by induction on $m_\mathbf{B}(w)$.

If $m_\mathbf{B}(e)=0$, we get $w=\mathrm{id}$, and thus $\ell_\mathbf{B}(w)=0$.
Now let $w\in\mathbf{W}$ such that $0<m_\mathbf{B}(w)<+\infty$ and assume that $\ell_\mathbf{B}(w')\leq m_\mathbf{B}(w')$ holds for all $w'\in\mathbf{W}$ such that $m_\mathbf{B}(w')<m_\mathbf{B}(w)$.
Let $e\in E'$ be minimal such that there is $e'\in E'$ with $e\prec_\mathbf{B} e'$ and $w(e)\succ_\mathbf{B}w(e')$.\ Choose $e'$ maximal for this property. We claim that
\begin{equation}
\label{claim-15}
\mbox{the set $\{i\in E:w(e)\succ_\mathbf{B}i\succ_\mathbf{B}w(e')\}$ is finite.}
\end{equation}Assume the contrary. Since $w$ fixes all but finitely many elements of $E$, there are infinitely many $i\in E$ such that $w(e)\succ_\mathbf{B}i=w(i)\succ_\mathbf{B}w(e')$.
Note that we have $e\prec_\mathbf{B}i$ by the minimality of $e$. Thus there are infinitely many elements in the set $I_e(w)$. In view of (\ref{5-2}), this is impossible,\ and (\ref{claim-15}) is established.

By (\ref{claim-15}) we can find $i\in E'$ such that $w(e')\prec_\mathbf{B}i$ and $w(e'),i$ are consecutive in $E'$. Choose $e''\in E'$ such that $i=w(e'')$. By the maximality of $e'$, we have $e''\prec_\mathbf{B}e'$. In the case where $\mathbf{G}=\mathbf{G}^\omega(E)$, up to replacing $(e'',e')$ by $(i_E(e'),i_E(e''))$ if necessary, we may assume that $e''\prec_\mathbf{B}i_E(e'')$. Hence we have found $e'',e'\in E'$ with the following properties:
\[e''\prec_\mathbf{B}e';\quad \mbox{$w(e')\prec_\mathbf{B}w(e'')$ are consecutive in $E'$};\quad e''\prec_\mathbf{B}i_E(e'')\mbox{ (in the case where $\mathbf{G}=\mathbf{G}^\omega(E)$)}.\]
It is straightforward to deduce that $m_\mathbf{B}(s_{w(e'),w(e'')}w)=m_\mathbf{B}(w)-1$.
Using the induction hypothesis, we derive: $\ell_\mathbf{B}(w)\leq \ell_\mathbf{B}(s_{w(e'),w(e'')}w)+1\leq m_\mathbf{B}(s_{w(e'),w(e'')}w)+1=m_\mathbf{B}(w)$.\ The proof is now complete.
\end{proof}

\begin{corollary}
\label{corollary-8-new}
The following conditions are equivalent.
\begin{itemize}
\item[\rm (i)] $\mathcal{S}_\mathbf{B}$ generates $\mathbf{W}$;
\item[\rm (ii)] $\ell_\mathbf{B}(w)<+\infty$ for all $w\in \mathbf{W}$;
\item[\rm (iii)] $(E,\preceq_\mathbf{B})$ is isomorphic as an ordered set to a subset of $(\mathbb{Z},\leq)$.
\end{itemize}
\end{corollary}

\begin{proof}
The equivalence {\rm (i)}$\Leftrightarrow${\rm (ii)} is immediate.
Note that condition {\rm (iii)} is equivalent to requiring that,
for all $e,e'\in E$, the interval $\{e''\in E:e\prec_\mathbf{B}e''\prec_\mathbf{B}e'\}$ is finite.
Thus the implication {\rm (iii)}$\Rightarrow${\rm (ii)} is guaranteed by Proposition \ref{proposition-8-new}\,{\rm (c)}. Conversely, if {\rm (ii)} holds true, then we get $\ell_\mathbf{B}(s_{e,e'})<+\infty$ for all $(e,e')\in \hat{E}$, whence (by Proposition \ref{proposition-8-new}\,{\rm (c)}) the set $\{e''\in E:e\prec_\mathbf{B}e''\prec_\mathbf{B}e'\}$ is finite.
This implies {\rm (iii)}.
\end{proof}

\subsubsection{Relation with parabolic subgroups}
In addition to the splitting Borel subgroup $\mathbf{B}$, we consider a splitting parabolic subgroup $\mathbf{P}\subset\mathbf{G}$ containing $\mathbf{H}$. Recall that the subgroup $\mathbf{P}$ gives rise (in fact, is equivalent) to each of the following data:
\begin{itemize}
\item an $E$-compatible generalized flag $\mathcal{F}$ (which is $\omega$-isotropic in the case of $\mathbf{G}=\mathbf{G}^\omega(E)$) such that $\mathbf{P}=\{g\in \mathbf{G}:g(\mathcal{F})=\mathcal{F}\}$;
\item a totally ordered set $(A,\preceq_A)$ and a surjective map $\sigma_0:E\to A$ such that $\mathcal{F}=\mathcal{F}_{\sigma_0}$ (which is equipped with an anti-automorphism $i_A:A\to A$ satisfying $\sigma_0\circ i_E=i_A\circ\sigma_0$ in the case of $\mathbf{G}=\mathbf{G}^\omega(E)$);
\item a partial order $\preceq_\mathbf{P}$ on $E$ satisfying property (\ref{10-new}), such that $e\prec_\mathbf{P}e'$ if and only if $\sigma_0(e)\prec_A\sigma_0(e')$.
\end{itemize}
Moreover, $\mathbf{P}$ gives rise to a subgroup of $\mathbf{W}$:
\[\mathbf{W}_\mathbf{P}=\{w\in\mathbf{W}:\sigma_0\circ w^{-1}=\sigma_0\}=\{w\in\mathbf{W}:e\not\prec_\mathbf{P}w(e)\mbox{ and }w(e)\not\prec_\mathbf{P}e,\ \forall e\in E\}.\]
Note that we do not assume that $\mathbf{B}$ is contained in $\mathbf{P}$.

\begin{lemma}
\label{lemma-4-new}
The following conditions are equivalent:
\begin{itemize}
\item[\rm (i)] $\mathbf{B}\subset\mathbf{P}$;
\item[\rm (ii)] for all $e,e'\in E$, $e\prec_\mathbf{P}e'\ \Rightarrow\ e\prec_\mathbf{B}e'$, i.e., the total order $\preceq_\mathbf{B}$ refines the partial order $\preceq_\mathbf{P}$;
\item[\rm (iii)] for all $e,e'\in E$, $e\preceq_\mathbf{B}e'\ \Rightarrow\ \sigma_0(e)\preceq_A\sigma_0(e')$,
i.e., the map $\sigma_0$ is nondecreasing.
\end{itemize}
\end{lemma}

\begin{proof}
By the definition of the generalized flag $\mathcal{F}_{\sigma_0}$, conditions {\rm (i)} and {\rm (iii)} are equivalent. Since the relation $e\prec_\mathbf{P}e'$ is equivalent to $\sigma_0(e')\not\preceq_A\sigma_0(e)$, we obtain that {\rm (ii)} and {\rm (iii)} are equivalent.
\end{proof}

For all $w\in\mathbf{W}$, we let
\begin{eqnarray*}
m_\mathbf{B}^\mathbf{P}(w)
 & = & \left\{\begin{array}{ll}
|\{(e,e')\in\hat{E}:e\prec_\mathbf{B}e',\ w(e)\succ_\mathbf{P} w(e')\}|  & \mbox{if $\mathbf{G}=\mathbf{G}(E)$} \\
|\{(e,e')\in\hat{E}:e\prec_\mathbf{B}e',\ e\prec_\mathbf{B}i_E(e),\ w(e)\succ_\mathbf{P} w(e')\}| & \mbox{if $\mathbf{G}=\mathbf{G}^\omega(E)$.}
\end{array}\right.
\end{eqnarray*}
Note that
\begin{equation}
\label{16-newnew}
m_\mathbf{B}^\mathbf{P}(w)=\left\{\begin{array}{ll}
n_\mathrm{inv}(\sigma_0\circ w) & \mbox{if $\mathbf{G}=\mathbf{G}(E)$} \\
n^\omega_\mathrm{inv}(\sigma_0\circ w) & \mbox{if $\mathbf{G}=\mathbf{G}^\omega(E)$}
\end{array}\right.
\end{equation}
(see Sections \ref{section-4-1}--\ref{section-4-2}).
We also know that $m_\mathbf{B}^\mathbf{B}(w)=\ell_\mathbf{B}(w)$ (see Proposition \ref{proposition-8-new}\,{\rm (b)}).
In the following proposition, we characterize the property that $\mathbf{B}$ is conjugate to a subgroup of $\mathbf{P}$ in terms of  $m_\mathbf{B}^\mathbf{P}(w)$.

\begin{proposition}
\label{proposition-9-new}
For $w\in\mathbf{W}$, recall that $\hat{w}\in\mathbf{G}$ is a representative of $w$ in $N_\mathbf{G}(\mathbf{H})$.
\begin{itemize}
\item[{\rm (a)}] We have $\mathbf{B}\subset \hat{w}\mathbf{P}\hat{w}^{-1}$ if and only if $m_\mathbf{B}^\mathbf{P}(w^{-1})=0$.
\item[{\rm (b)}] The following conditions are equivalent:
\begin{itemize}
\item[\rm (i)] there is $w\in\mathbf{W}$ such that $\mathbf{B}\subset \hat{w}\mathbf{P}\hat{w}^{-1}$;
\item[\rm (ii)] there is $w\in\mathbf{W}$ such that $m_\mathbf{B}^\mathbf{P}(w)<+\infty$.
\end{itemize}
\end{itemize}
\end{proposition}

\begin{proof}
Note that $\hat{w}\mathbf{P}\hat{w}^{-1}\subset\mathbf{G}$ is the isotropy subgroup of the generalized flag $\mathcal{F}_{\sigma_0\circ w^{-1}}$.
Thus part {\rm (a)} follows from Lemma \ref{lemma-4-new} and the definition of  $m_\mathbf{B}^\mathbf{P}(w^{-1})$. \\
{\rm (b)} The implication {\rm (i)}$\Rightarrow${\rm (ii)} follows from part {\rm (a)}. Now assume that {\rm (ii)} holds. Choose $w\in\mathbf{W}$ such that $m_\mathbf{B}^\mathbf{P}(w)$ is minimal.
By {\rm (a)}, it suffices to show that $m_\mathbf{B}^\mathbf{P}(w)=0$.
Assume, to the contrary, that $m_\mathbf{B}^\mathbf{P}(w)>0$.
Hence there is a couple $(e,e')\in\hat{E}$ satisfying $e\prec_\mathbf{B}e'$, $w(e)\succ_\mathbf{P}w(e')$. We can assume that $e$ is minimal such that there is $e'$ with this property, and that $e'$ is maximal possible.
We claim that
\begin{equation}
\label{claim-16}
\mbox{the set $\{i\in E':w(e)\succ_\mathbf{P}i\succ_\mathbf{P}w(e')\}$ is finite.}
\end{equation}
Otherwise,\ there are infinitely many $i\in E$ for which $w(e)\succ_\mathbf{P}i=w(i)\succ_\mathbf{P}w(e')$. By the minimality of $e$, we have $e\prec_\mathbf{B}i$. Whence there are infinitely many couples $(e,i)\in\hat{E}$ with $e\prec_\mathbf{B}i$ and $w(e)\succ_\mathbf{P}w(i)$ (in the case of $\mathbf{G}=\mathbf{G}^\omega(E)$, up to replacing $(e,i)$ by $(i_E(i),i_E(e))$, we may also assume that $e\prec_\mathbf{B}i_E(e)$). Consequently, $m_\mathbf{B}^\mathbf{P}(w)=+\infty$, a contradiction. This establishes (\ref{claim-16}).

By (\ref{claim-16}) we can find $i\in E'$ minimal (with respect to the order $\preceq_\mathbf{P}$) such that $w(e)\succeq_\mathbf{P}i\succ_\mathbf{P}w(e')$. Let $e''\in E$ with $w(e'')=i$. The maximality of $e'$ forces $e''\prec_\mathbf{B}e'$.
Altogether, we have found a couple $(e'',e')\in\hat{E}$ such that $e''\prec_\mathbf{B}e'$, $w(e'')\succ_\mathbf{P}w(e')$, and $w(e'')$ is minimal  (with respect to the order $\preceq_\mathbf{P}$).
For $f\in E$, let $C_\mathbf{P}(f)$ denote the class of $f$ for the equivalence relation defined in (\ref{10-new}).
We may assume that $e''$ and $e'$ are respectively a minimal element of $w^{-1}(C_\mathbf{P}(w(e'')))$ and a maximal element of $w^{-1}(C_\mathbf{P}(w(e')))$ (with respect to the order $\preceq_\mathbf{B}$).
Moreover, in the case of $\mathbf{G}=\mathbf{G}^\omega(E)$, up to replacing $(e'',e')$ by $(i_E(e'),i_E(e''))$, we may assume that $e''\prec_\mathbf{B}i_E(e'')$.
Then it is straightforward to check that
\begin{eqnarray*}
  \{(f,f')\in\hat{E}:f\prec_\mathbf{B}f',\ s_{w(e'),w(e'')}w(f)\succ_\mathbf{P} s_{w(e'),w(e'')}w(f')\} \\ \subset \{(f,f')\in\hat{E}: f\prec_\mathbf{B}f',\ w(f)\succ_\mathbf{P}w(f')\}\setminus\{(e'',e')\}.
\end{eqnarray*}
Whence $m_\mathbf{B}^\mathbf{P}(s_{w(e'),w(e'')}w)<m_\mathbf{B}^\mathbf{P}(w)$, which contradicts the minimality of $m_\mathbf{B}^\mathbf{P}(w)$.
\end{proof}

Finally, the following proposition points out the relation between  $m_\mathbf{B}^\mathbf{P}(w)$ and  $\ell_\mathbf{B}(w)$.

\begin{proposition}
\label{proposition-10-new}
Assume that there is $w_0\in\mathbf{W}$ such that $m_\mathbf{B}^\mathbf{P}(w_0^{-1})=0$.
Then, for all $w\in\mathbf{W}$, we have
\[m_\mathbf{B}^\mathbf{P}(w)=\inf\{\ell_\mathbf{B}(w_0w'w):w'\in\mathbf{W}_\mathbf{P}\}.\]
\end{proposition}

\begin{proof}
Note that, for all $e,e'\in E'$, we have $e\prec_\mathbf{P}e'$ if and only if $w_0(e)\prec_{\hat{w_0}\mathbf{P}\hat{w_0}^{-1}}w_0(e')$.
This yields $m_\mathbf{B}^\mathbf{P}(w)=m_\mathbf{B}^{\hat{w_0}\mathbf{P}\hat{w_0}^{-1}}(w_0w)$
and $w_0\mathbf{W}_{\mathbf{P}}w_0^{-1}=\mathbf{W}_{\hat{w_0}\mathbf{P}\hat{w_0}^{-1}}$. Thus, invoking also Proposition \ref{proposition-9-new}\,{\rm (a)}, up to replacing $\mathbf{P}$ by $\hat{w_0}\mathbf{P}\hat{w_0}^{-1}$, we may suppose that $\mathbf{B}\subset\mathbf{P}$ and $w_0=\mathrm{id}$.

By the definition of $\mathbf{W}_\mathbf{P}$, Lemma \ref{lemma-4-new}, and Proposition \ref{proposition-8-new}\,{\rm (b)}, for every $w'\in\mathbf{W}_\mathbf{P}$ we obtain
\begin{eqnarray*}
m_\mathbf{B}^\mathbf{P}(w) & = & |\{(e,e')\in\hat{E}_\mathbf{B}:\sigma_0(w(e))\succ_A\sigma_0(w(e'))\}| \\
 & = & |\{(e,e')\in\hat{E}_\mathbf{B}:\sigma_0(w'w(e))\succ_A \sigma_0(w'w(e))\}| \\
 & \leq & |\{(e,e')\in\hat{E}_\mathbf{B}:w'w(e)\succ_\mathbf{B}w'w(e')\}|=\ell_\mathbf{B}(w'w), \end{eqnarray*}
where $\hat{E}_\mathbf{B}=\{(e,e')\in\hat{E}:e\prec_\mathbf{B}e'\}$ if $\mathbf{G}=\mathbf{G}(E)$,
and $\hat{E}_\mathbf{B}=\{(e,e')\in\hat{E}:e\prec_\mathbf{B}e',\ e\prec_\mathbf{B}i_E(e)\}$ if $\mathbf{G}=\mathbf{G}^\omega(E)$.
If $m_\mathbf{B}^\mathbf{P}(w)=+\infty$, the result is established.
So we assume next that $m_\mathbf{B}^\mathbf{P}(w)<+\infty$.

\smallskip
\noindent
{\it Claim 1:} There is $w'\in\mathbf{W}_\mathbf{P}$ such that the set $\mathcal{I}(w'w):=\{e\in E:\sigma_0(e)=\sigma_0(w'w(e))$ and $w'w(e)\not=e\}$ is empty.

For any $w'\in\mathbf{W}_\mathbf{P}$, the set $\mathcal{I}(w'w)$ is finite. Let $w'\in\mathbf{W}_\mathbf{P}$ such that $|\mathcal{I}(w'w)|$ is minimal. We claim that $\mathcal{I}(w'w)=\emptyset$. For otherwise, assume that there is $e\in \mathcal{I}(w'w)$. Thus $\sigma_0(w'w(e))=e$.
Either $\sigma_0((w'w)^\ell(e))=\sigma_0(e)$ for all $\ell\in\mathbb{Z}$, or there is $\ell\in\mathbb{Z}$ such that $\sigma_0((w'w)^{\ell-1}(e))\not=\sigma_0((w'w)^\ell(e))=\sigma_0((w'w)^{\ell+1}(e))$.
In the former case we set $w''=s_{(w'w)^{m-2}(e),(w'w)^{m-1}(e)}\cdots$ $\cdots s_{(w'w)(e),(w'w)^2(e)}s_{e,(w'w)(e)}$,
where $m\geq 2$ is minimal such that $(w'w)^m(e)=e$.
In the latter case we set $w''=s_{(w'w)^{\ell}(e),(w'w)^{\ell+1}(e)}$.
In both cases one has $w''\in\mathbf{W}_{\mathbf{P}}$, and it easy to check that $\mathcal{I}(w''w'w)\varsubsetneq \mathcal{I}(w'w)$, a contradiction. Hence Claim 1 holds.

\smallskip

Note that $m_\mathbf{B}^\mathbf{P}(w'w)=m_\mathbf{B}^\mathbf{P}(w)$. Up to dealing with $w'w$ instead of $w$, we may assume that $\mathcal{I}(w)=\emptyset$.
For $\alpha\in A$, let
$I_\alpha(w)=\{e\in \sigma_0^{-1}(\alpha):w(e)\not=e\}$. Since $\mathcal{I}(w)=\emptyset$, one has $I_\alpha(w)=I_\alpha^+(w)\sqcup I_\alpha^-(w)$ with
\[I_\alpha^+(w)=\{e\in \sigma_0^{-1}(\alpha):\sigma_0(w^{-1}(e))\succ_A \alpha\}\quad\mbox{and}\quad I_\alpha^-(w)=\{e\in \sigma_0^{-1}(\alpha):\sigma_0(w^{-1}(e))\prec_A \alpha\}.\]

\smallskip
\noindent
{\it Claim 2:}
There is $w'\in\mathbf{W}_\mathbf{P}$ with $w'(e)=e$ whenever $w(e)=e$, and satisfying the following property:
for every $\alpha\in A$, the set $\{e'\in\sigma_0^{-1}(\alpha):w'(e)\prec_\mathbf{B}e'\}$ is finite whenever $e\in I_\alpha^+(w)$, and the set $\{e'\in\sigma_0^{-1}(\alpha):w'(e)\succ_\mathbf{B}e'\}$ is finite whenever $e\in I_\alpha^-(w)$.

Let $e\in I_\alpha^+(w)$. There is $\ell(e)\geq 2$ minimal such that $\sigma_0(w^{-\ell(e)}(e))\preceq_A\alpha$. Since $m_\mathbf{B}^\mathbf{P}(w)<+\infty$, the set $\{e'\in\sigma_0^{-1}(\alpha):w^{-\ell(e)}(e)\prec_\mathbf{B}e'\}$ is finite. Set $w'(e)=w^{-\ell(e)}(e)$. Similarly, given $e\in I_\alpha^-(w)$, there is $m(e)\geq 2$ minimal such that $\sigma_0(w^{-m(e)}(e))\succeq_A\alpha$, and the set $\{e'\in\sigma_0^{-1}(\alpha):w^{-m(e)}(e)\succ_\mathbf{B}e'\}$ is finite; we set $w'(e)=w^{-m(e)}(e)$ in this case. If $e\in \sigma_0^{-1}(\alpha)\setminus I_\alpha(w)$, we set $w'(e)=e$. It is readily seen that the so-obtained map $w':\sigma_0^{-1}(\alpha)\to\sigma_0^{-1}(\alpha)$ is bijective. Collecting these maps for all $\alpha\in A$, we obtain an element $w'\in\mathbf{W}_\mathbf{P}$ satisfying the desired properties. This shows Claim 2.

\smallskip

Set $\hat{w}=w'w$ with $w'\in\mathbf{W}_\mathbf{P}$ as in Claim 2.
For every $\alpha\in A$, the set
\begin{eqnarray*}
J_\alpha(\hat{w})=\{e\in\sigma_0^{-1}(\alpha):(\exists e'\in\sigma_0^{-1}(\alpha)\mbox{ with }e'\preceq_\mathbf{B}e\mbox{ and }\sigma_0(\hat{w}^{-1}(e'))\succ_A\alpha) \\ \mbox{ or }(\exists e'\in\sigma_0^{-1}(\alpha)\mbox{ with }e'\succeq_\mathbf{B}e\mbox{ and }\sigma_0(\hat{w}^{-1}(e'))\prec_A\alpha)\}
\end{eqnarray*}
is finite (by Claim 2). We write $J_\alpha(\hat{w})=\{e_i^\alpha\}_{i=1}^{k_\alpha}$
so that $\hat{w}^{-1}(e_1^\alpha)\prec_\mathbf{B}\ldots\prec_\mathbf{B}\hat{w}^{-1}(e_{k_\alpha}^\alpha)$.
There is $w''\in\mathbf{W}_\mathbf{P}$ with $w''(e)=e$ whenever $e\notin\bigcup_{\alpha\in A}J_\alpha(\hat{w})$ and such that
\[w''(J_\alpha(\hat{w}))=J_\alpha(\hat{w})\ \ \mbox{and}\ \ w''(e_1^\alpha)\prec_\mathbf{B}\ldots\prec_\mathbf{B}w''(e_{k_\alpha}^\alpha)\ \mbox{ for all $\alpha\in A$.}\]
Taking the construction of $w''$ into account, one can check that there is no couple $(e,e')\in\hat{E}$ with $e\prec_\mathbf{B}e'$, $w''\hat{w}(e)\succ_\mathbf{B}w''\hat{w}(e')$, and $\sigma_0(w''\hat{w}(e))=\sigma_0(w''\hat{w}(e'))$. Therefore, $m_\mathbf{B}^\mathbf{P}(w)=\ell_\mathbf{B}(w''\hat{w})=\ell_\mathbf{B}((w''w')w)$ with $w''w'\in\mathbf{W}_\mathbf{P}$. The proof is complete.
\end{proof}

\subsection{Review of (finite-dimensional) flag varieties}
\label{section-5-2}
We consider an $E$-compatible generalized flag $\mathcal{F}=\mathcal{F}_{\sigma_0}$
corresponding to a surjection $\sigma_0:E\to A$. Let $I\subset E$ be a finite subset (resp., $i_E$-stable, if the form $\omega$ is considered).
In this section we recall standard properties of the Schubert decomposition of the flag varieties $\mathrm{Fl}(\mathcal{F},I)$ and $\mathrm{Fl}(\mathcal{F},\omega,I)$ (see Section \ref{section-3-3}). We refer to \cite{Billey-Lakshmibai,Brion,Lak-book} for more details.

\begin{proposition}
\label{proposition-11-new}
Let $\mathbf{G}=\mathbf{G}(E)$.
Let $\mathbf{B}$ be a splitting Borel subgroup of $\mathbf{G}$ containing $\mathbf{H}$ and let $B(I):=G(I)\cap\mathbf{B}$ be the corresponding Borel subgroup of the group $G(I)$. Let $H(I)=G(I)\cap\mathbf{H}$. Let $W(I)\subset\mathbf{W}$ be the Weyl group of $G(I)$. \\
{\rm (a)} We have the decomposition
\[\mathrm{Fl}(\mathcal{F},I)=\bigcup_{w\in W(I)}B(I)\mathcal{F}_{\sigma_0\circ w^{-1}}.\]
Moreover, $\mathcal{F}_{\sigma_0\circ w^{-1}}$ is the unique element of $B(I)\mathcal{F}_{\sigma_0\circ w^{-1}}$ fixed by the maximal torus $H(I)$. \\
{\rm (b)} Each subset $B(I)\mathcal{F}_{\sigma_0\circ w^{-1}}$, for $w\in W(I)$, is a locally closed subvariety isomorphic to an affine space of dimension $|\{(e,e')\in I\times I:e\prec_\mathbf{B}e',\ \sigma_0\circ w^{-1}(e')\prec_A\sigma_0\circ w^{-1}(e)\}|$. \\
{\rm (c)} Given $w,w'\in W(I)$, the inclusion $B(I)\mathcal{F}_{\sigma_0\circ w^{-1}}\subset
\overline{B(I)\mathcal{F}_{\sigma_0\circ w'^{-1}}}$ holds if and only if $\sigma_0\circ w^{-1}\leq\sigma_0\circ w'^{-1}$ for the order $\leq$ defined in Section \ref{section-4-1}. \\
{\rm (d)} Let $J\subset E$ be another finite subset such that $I\subset J$. Let $\iota_{I,J}:\mathrm{Fl}(\mathcal{F},I)\hookrightarrow\mathrm{Fl}(\mathcal{F},J)$ be the embedding constructed in Section \ref{section-3-3}.
Then, for all $w\in W(I)$, the image of the Schubert cell $B(I)\mathcal{F}_{\sigma_0\circ w^{-1}}$ by the map $\iota_{I,J}$ is an affine subspace of $B(J)\mathcal{F}_{\sigma_0\circ w^{-1}}$.
\end{proposition}

\begin{proposition}
\label{proposition-12-new}
Let $\mathbf{G}=\mathbf{G}^\omega(E)$. Let $\mathbf{B}$ be a splitting Borel subgroup of $\mathbf{G}$ containing $\mathbf{H}$ and let $B^\omega(I):=G^\omega(I)\cap\mathbf{B}$ be the corresponding Borel subgroup of the group $G^\omega(I)$. Let $H^\omega(I)=G^\omega(I)\cap \mathbf{H}$. Let $W^\omega(I)\subset\mathbf{W}$ be the Weyl group of $G^\omega(I)$. \\
{\rm (a)} We have the decomposition
\[\mathrm{Fl}(\mathcal{F},\omega,I)=\bigcup_{w\in W^\omega(I)}B^\omega(I)\mathcal{F}_{\sigma_0\circ w^{-1}}.\]
Moreover, $\mathcal{F}_{\sigma_0\circ w^{-1}}$ is the unique element of $B^\omega(I)\mathcal{F}_{\sigma_0\circ w^{-1}}$ fixed by the maximal torus $H^\omega(I)$. \\
{\rm (b)} Each subset $B^\omega(I)\mathcal{F}_{\sigma_0\circ w^{-1}}$, for $w\in W^\omega(I)$, is a locally closed subvariety isomorphic to an affine space of dimension $|\{(e,e')\in I\times I:e\prec_\mathbf{B}e',\ e\prec_\mathbf{B}i_E(e),\ e'\not=i_E(e'),\ \sigma_0\circ w^{-1}(e')\prec_A\sigma_0\circ w^{-1}(e)\}|$. \\
{\rm (c)} Given $w,w'\in W^\omega(I)$, the inclusion $B^\omega(I)\mathcal{F}_{\sigma_0\circ w^{-1}}\subset
\overline{B^\omega(I)\mathcal{F}_{\sigma_0\circ w'^{-1}}}$ holds if and only if $\sigma_0\circ w^{-1}\leq_\omega\sigma_0\circ w'^{-1}$, for the order $\leq_\omega$ defined in Section \ref{section-4-2}. \\
{\rm (d)} Let $J\subset E$ be another $i_E$-stable finite subset such that $I\subset J$. Let $\iota^\omega_{I,J}:\mathrm{Fl}(\mathcal{F},\omega,I)\hookrightarrow\mathrm{Fl}(\mathcal{F},\omega,J)$ be the embedding constructed in Section \ref{section-3-3}.
Then, for all $w\in W^\omega(I)$, the image of the Schubert cell $B^\omega(I)\mathcal{F}_{\sigma_0\circ w^{-1}}$ by the map $\iota^\omega_{I,J}$ is an affine subspace of $B^\omega(J)\mathcal{F}_{\sigma_0\circ w^{-1}}$.
\end{proposition}

\subsection{Proof of Lemmas \ref{lemma-2-new} and \ref{lemma-3}}

\label{section-5-3}

We consider the map
\[\phi:\mathbf{W}(E)\to\mathbf{Fl}_A(V),\ w\mapsto\mathcal{F}_{\sigma_0\circ w^{-1}}\]
and, in the proof of Lemma \ref{lemma-3}, we also consider its restriction $\phi^\omega:\mathbf{W}^\omega(E)\to\mathbf{Fl}_A^\omega(V)$.

\begin{proof}[Proof of Lemma \ref{lemma-2-new}]
Let $\mathbf{Fl}'(\mathcal{F},E)\subset\mathbf{Fl}(\mathcal{F},E)$ denote the subset of $E$-compatible generalized flags.
By definition the generalized flag $\phi(w)$ is $E$-compatible for all $w\in\mathbf{W}(E)$. Moreover, it is easily seen that $\phi(w)=\hat{w}(\mathcal{F}_{\sigma_0})$ where $\hat{w}\in\mathbf{G}(E)$ is the element for which $\hat{w}(e)=w(e)$ for all $e\in E$. Thus $\phi(w)$ is $E$-commensurable with $\mathcal{F}=\mathcal{F}_{\sigma_0}$ (see Proposition \ref{proposition-2}). Consequently, $\phi(w)\in\mathbf{Fl}'(\mathcal{F},E)$ for all $w\in\mathbf{W}(E)$.

Conversely, let $\mathcal{G}\in\mathbf{Fl}'(\mathcal{F},E)$. Choosing $n$ such that $\mathcal{G}\in\mathrm{Fl}(\mathcal{F},E_n)$, we have that $\mathcal{G}$ is fixed by the maximal torus $H(E_n)\subset G(E_n)$. Using Proposition \ref{proposition-11-new}\,{\rm (a)}, we
find $w\in W(E_n)\subset\mathbf{W}(E)$ such that $\mathcal{G}=\mathcal{F}_{\sigma_0\circ w^{-1}}=\phi(w)$.

Finally, for $w,w'\in\mathbf{W}(E)$, we have $\phi(w)=\phi(w')$ if and only if $\sigma_0\circ w^{-1}=\sigma_0\circ w'^{-1}$, and the latter condition reads as $w'^{-1}w\in\mathbf{W}_\mathbf{P}(E)$. Therefore, $\phi$ induces a bijection $\mathbf{W}(E)/\mathbf{W}_\mathbf{P}(E)\to\mathbf{Fl}'(\mathcal{F},E)$.
\end{proof}

\begin{proof}[Proof of Lemma \ref{lemma-3}]
Let $\mathbf{Fl}'(\mathcal{F},\omega,E)=\mathbf{Fl}'(\mathcal{F},E)\cap \mathbf{Fl}(\mathcal{F},\omega,E)$.
From Lemma \ref{lemma-2-new} we know that $\phi^\omega(w)$ is $E$-compatible and $E$-commensurable with $\mathcal{F}=\mathcal{F}_{\sigma_0}$, whence $\phi^\omega(w)\in\mathbf{Fl}'(\mathcal{F},\omega,E)$ for all $w\in\mathbf{W}^\omega(E)$.

Let $\mathcal{G}\in\mathbf{Fl}'(\mathcal{F},\omega,E)$. Choosing $n$ such that $\mathcal{G}\in \mathrm{Fl}(\mathcal{F},\omega,E_n)$, we have that $\mathcal{G}$ is a fixed point of the maximal torus $H^\omega(E_n)\subset G^\omega(E_n)$, hence we can find $w\in W^\omega(E_n)$ such that $\mathcal{G}=\mathcal{F}_{\sigma_0\circ w^{-1}}=\phi^\omega(w)$.

As in the proof of Lemma \ref{lemma-2-new} it is easy to conclude that $\phi^\omega$ induces a bijection $\mathbf{W}^\omega(E)/\mathbf{W}_\mathbf{P}^\omega(E)\to\mathbf{Fl}'(\mathcal{F},\omega,E)$.
\end{proof}

\subsection{Proof of Theorems \ref{theorem-1} and \ref{theorem-2}}

\label{section-5-4}

\begin{proof}[Proof of Theorem \ref{theorem-1}]
Recall the exhaustions (\ref{rel3}) and (\ref{rel8-newnew}) of the ind-group $\mathbf{G}(E)$ and the ind-variety $\mathbf{Fl}(\mathcal{F},E)$.
For all $n\geq 1$, the subgroups $H(E_n):=G(E_n)\cap\mathbf{H}(E)$, $B_n:=G(E_n)\cap\mathbf{B}$, and $P_n:=G(E_n)\cap\mathbf{P}$ are respectively a maximal torus, a Borel subgroup, and a parabolic subgroup of $G(E_n)$.

{\rm (a)} Let $\mathcal{G}\in\mathbf{Fl}(\mathcal{F},E)$. By Proposition \ref{proposition-11-new}\,{\rm (a)}, for any $n\geq 1$ large enough so that $\mathcal{G}\in\mathrm{Fl}(\mathcal{F},E_n)$,
the $B_n$-orbit of $\mathcal{G}$ contains a unique element of the form $\mathcal{F}_{\sigma_0\circ w^{-1}}$ with $w\in W(E_n)$.
Therefore, every element $\mathcal{G}\in\mathbf{Fl}(\mathcal{F},E)$ lies in the $\mathbf{B}$-orbit of $\mathcal{F}_\sigma$ for a unique $\sigma\in\mathbf{W}(E)\cdot\sigma_0$.

{\rm (b)}
Let $\mathcal{G}=\{G'_\alpha,G''_\alpha:\alpha\in A\}\in\mathbf{Fl}(\mathcal{F},E)$.
According to part {\rm (a)} of the proof, there is a unique $\sigma\in\mathbf{W}(E)\cdot\sigma_0$ such that $\mathcal{G}\in\mathbf{B}\mathcal{F}_\sigma$, say $\mathcal{G}=b(\mathcal{F}_\sigma)$, where $b\in\mathbf{B}$. Thus
\[G''_\alpha\cap F'_{0,e}=b(F''_{\sigma,\alpha}\cap F'_{0,e})\quad\mbox{and}\quad G''_\alpha\cap F''_{0,e}=b(F''_{\sigma,\alpha}\cap F''_{0,e})\]
(because $F'_{0,e},F''_{0,e}$ are $b$-stable). This clearly implies that $\sigma_\mathcal{G}=\sigma_{\mathcal{F}_\sigma}$.
Moreover, from the definition of $\mathcal{F}_\sigma$ we see that $F''_{\sigma,\alpha}\cap F''_{0,e}\not= F''_{\sigma,\alpha}\cap F'_{0,e}$ if and only if $\sigma(e)\preceq_A\alpha$.
Whence $\sigma(e)=\min\{\alpha\in A:F''_{\sigma,\alpha}\cap F''_{0,e}\not= F''_{\sigma,\alpha}\cap F'_{0,e}\}=\sigma_{\mathcal{F}_\sigma}(e)$ for all $e\in E$. Thus $\sigma_{\mathcal{G}}=\sigma$. Note that the last equality guarantees in particular that $\sigma_\mathcal{G}\in\mathfrak{S}(E,A)$.

{\rm (c)}
follows from Proposition \ref{proposition-11-new}\,{\rm (b)} and {\rm (d)}.

{\rm (d)}
We consider $\sigma,\tau\in\mathbf{W}(E)\cdot\sigma_0$
and let $n\geq 1$ be such that $\mathcal{F}_\sigma,\mathcal{F}_\tau\in\mathrm{Fl}(\mathcal{F},E_n)$.
Assume that $\sigma\hat{<}\tau$, i.e., $\tau=\sigma\circ t_{e,e'}$ for a pair $(e,e')\in E\times E$ with $e\prec_\mathbf{B}e'$ and $\sigma(e)\preceq_A\sigma(e')$. Up to choosing $n$ larger if necessary, we may assume that $e,e'\in E_n$. Then, by Proposition \ref{proposition-11-new}\,{\rm (c)}, we get $B_n\mathcal{F}_\sigma\subset\overline{B_n\mathcal{F}_\tau}$.\ Whence
$\mathbf{B}\mathcal{F}_\sigma\subset\overline{\mathbf{B}\mathcal{F}_\tau}$.
This argument also shows that the latter inclusion holds whenever $\sigma\leq\tau$.
Conversely, assume that $\mathcal{F}_\sigma\in\overline{\mathbf{B}\mathcal{F}_\tau}$.
Hence  $\mathcal{F}_\sigma\in\overline{B_n\mathcal{F}_\tau}$ for $n\geq 1$ large enough. Once again, by Proposition \ref{proposition-11-new}\,{\rm (c)}, this yields $\sigma\leq\tau$. The proof of Theorem \ref{theorem-1} is complete.
\end{proof}

\begin{proof}[Proof of Theorem \ref{theorem-2}]
The proof of Theorem \ref{theorem-2} follows exactly the same scheme as the proof of Theorem \ref{theorem-1}, relying this time on Proposition \ref{proposition-12-new} instead of Proposition \ref{proposition-11-new}.
We skip the details.
\end{proof}

\subsection{Proof of Theorem \ref{theorem-3}}

\label{section-5-5}

\begin{proof}[Proof of Theorem \ref{theorem-3}]
{\rm (a)} Condition {\rm (i)} means that there is $g\in \mathbf{G}$ such that $\mathbf{B}\subset g\mathbf{P}g^{-1}$. This equivalently means that the element $g\mathbf{P}\in\mathbf{G}/\mathbf{P}$ is fixed by $\mathbf{B}$, i.e., that $\mathbf{G}/\mathbf{P}$ comprises a $\mathbf{B}$-orbit reduced to a single point. We have shown the equivalence {\rm (i)}$\Leftrightarrow${\rm (iii)}.
The implication {\rm (iii)}$\Rightarrow${\rm (ii)} is immediate, while
the implication {\rm (ii)}$\Rightarrow${\rm (i)} follows from Proposition \ref{proposition-9-new}, relation (\ref{16-newnew}), and Theorems \ref{theorem-1}{\rm (c)}--\ref{theorem-2}\,{\rm (c)}. \\
{\rm (b)} The implication {\rm (i)}$\Rightarrow${\rm (ii)} is a consequence of part {\rm (a)}, Corollary \ref{corollary-8-new}, Proposition \ref{proposition-10-new}, relation (\ref{16-newnew}), and Theorems \ref{theorem-1}\,{\rm (c)}--\ref{theorem-2}\,{\rm (c)}. Assume that {\rm (ii)} holds. From part (a), there is $g\in\mathbf{G}$ such that $\mathbf{B}\subset g\mathbf{P}g^{-1}$. Up to dealing with $g\mathbf{P}g^{-1}$ instead of $\mathbf{P}$, we may assume that $\mathbf{B}\subset\mathbf{P}$. Arguing by contradiction, say that $(E,\preceq_\mathbf{B})$ is not isomorphic to a subset of $(\mathbb{Z},\leq)$. Thus there are $e,e'\in E$ such that the set $\{e''\in E:e\prec_\mathbf{B}e''\prec_\mathbf{B}e'\}$ is infinite. Since the surjective map $\sigma_0:E\to A$, corresponding to $\mathbf{P}$, is nondecreasing (by Lemma \ref{lemma-4-new}) and nonconstant (because $\mathbf{P}\not=\mathbf{G}$), we find $\hat{e},\hat{e}'$ with $\hat{e}\preceq_\mathbf{B}e\prec_\mathbf{B}e'\preceq_\mathbf{B}\hat{e}'$ such that $\sigma_0(\hat{e})\prec_A\sigma_0(\hat{e}')$. Then, $\dim\mathbf{B}\mathcal{F}_{\sigma_0\circ s_{\hat{e},\hat{e}'}}=+\infty$ (by Theorems \ref{theorem-1}\,{\rm (c)}--\ref{theorem-2}\,{\rm (c)}), a contradiction.
\end{proof}

\section{Smoothness of Schubert ind-varieties}

In this section $\mathbf{G}$ is one of the ind-groups $\mathbf{G}(E)$ or $\mathbf{G}^\omega(E)$
and $\mathbf{B}$ is a splitting Borel subgroup of $\mathbf{G}$ which contains the
splitting Cartan subgroup $\mathbf{H}=\mathbf{H}(E)$ or $\mathbf{H}^\omega(E)$.
We consider the {\em Schubert ind-varieties} defined as the closures of the Schubert cells
$\mathbf{B}\mathcal{F}_\sigma$ in the ind-varieties of generalized flags $\mathbf{Fl}(\mathcal{F},E)$ or
$\mathbf{Fl}(\mathcal{F},\omega,E)$.
Specifically, we study the smoothness of Schubert ind-varieties.
The general principle (Theorem \ref{theorem-4}) is straightforward: the ind-variety $\overline{\mathbf{B}\mathcal{F}_\sigma}$ is smooth if and only if its  intersections with suitable finite-dimensional flag subvarieties
of $\mathbf{Fl}(\mathcal{F},E)$ or
$\mathbf{Fl}(\mathcal{F},\omega,E)$ are smooth. Note however that this fact is not immediate: see Remark \ref{remark-1} below. As an example, in Section \ref{section-4.2} we give a combinatorial interpretation of this result in the case of ind-varieties of maximal generalized flags and in the case of ind-grassmannians.

\subsection{General facts on the smoothness of ind-varieties}

\label{section-6-1}

The notion of smooth point of an ind-variety is defined in Section \ref{section-2-1}.
We refer to \cite[Chapter 4]{K} or \cite{S} for more details.
In this section, for later use, we present some general facts regarding the smoothness of ind-varieties.

We start with the following simple smoothness criterion (see \cite{K}).

\begin{lemma}
\label{lemma-4} Let $\mathbf{X}$ be an ind-variety
with an exhaustion $\mathbf{X}=\bigcup_{n\geq 1}X_n$.
Let $x\in \mathbf{X}$.
Suppose that there is a subsequence $\{X_{n_k}\}_{k\geq 1}$ such that $x$  is a smooth point of $X_{n_k}$ for all $k\geq 1$.
Then $x$ is a smooth point of $\mathbf{X}$. In particular, if $\mathbf{X}$ admits an exhaustion by smooth varieties, then $\mathbf{X}$ is smooth.
\end{lemma}

\begin{example}
It easily follows from Lemma \ref{lemma-4} that the infinite-dimensional affine space
$\mathbb{A}^\infty$ and the infinite-dimensional projective space $\mathbb{P}^\infty$ are
smooth. More generally,
it follows from Propositions \ref{proposition-3-3-3}--\ref{proposition-3-3-2} and Lemma \ref{lemma-4} that
the ind-varieties of the form $\mathbf{Fl}(\mathcal{F},E)$ and $\mathbf{Fl}(\mathcal{F},\omega,E)$ are smooth.
\end{example}

\begin{remark}
\label{remark-1}
The converse of Lemma \ref{lemma-4} is clearly false. Consider for
instance $\mathbf{X}=\mathbb{A}^\infty=\bigcup_{n\geq 1}\mathbb{A}^n$ and let $x\in\mathbb{A}^1$. For each $n\geq 1$, let $X'_n\subset\mathbb{A}^{n+1}$ be an $n$-dimensional affine subspace containing $x$ and distinct of $\mathbb{A}^n$, and set $X_n=\mathbb{A}^n\cup X'_n$. The subvarieties $X_n$ exhaust $\mathbb{A}^\infty$.
Clearly $x$ is a singular point of every $X_n$. However
$x$ is a smooth point of $\mathbb{A}^\infty$ (which is a smooth ind-variety).
\end{remark}

The following partial converse of Lemma \ref{lemma-4} is used in Section \ref{section-6-2} for studying the smoothness of Schubert ind-varieties.

\begin{lemma}
\label{lemma-2} Let $\mathbf{X}$ be an ind-variety and let $\mathbf{X}=\bigcup_{n\geq 1}X_n$ be an exhaustion by algebraic varieties. Assume that each inclusion $X_n\subset X_{n+1}$ has a
left inverse $r_n:X_{n+1}\to X_n$ in the category of algebraic varieties. Then, if $x\in \mathbf{X}$ is a singular point of $X_{n_0}$ for
some $n_0\geq 1$, $x$ is a singular point of $\mathbf{X}$.
\end{lemma}

\begin{proof}
We start with a preliminary fact.
Let $Y$ be an algebraic variety and $Z\subset Y$ be a subvariety such that there is a retraction $r:Y\to Z$, i.e., a left inverse of the inclusion map $i:Z\hookrightarrow Y$. Let $x\in Z$.
We consider the local rings $\mathcal{O}_{Z,x}$, $\mathcal{O}_{Y,x}$ and their maximal ideals $\mathfrak{m}_{Z,x}$, $\mathfrak{m}_{Y,x}$.
The map $r$ induces a ring homomorphism  $r^*:\mathcal{O}_{Z,x}\to\mathcal{O}_{Y,x}$ such that
 $r^*(\mathfrak{m}_{Z,x}^k)\subset\mathfrak{m}_{Y,x}^k$ for all $k\geq 1$.  Thus $r^*$ induces maps
\[r_{Z,k}:S^k(\mathfrak{m}_{Z,x}/\mathfrak{m}^2_{Z,x})\to S^k(\mathfrak{m}_{Y,x}/\mathfrak{m}^2_{Y,x})
\quad\mbox{and}\quad
\tilde{r}_{Z,k}:\mathfrak{m}^k_{Z,x}/\mathfrak{m}^{k+1}_{Z,x}\to \mathfrak{m}^k_{Y,x}/\mathfrak{m}^{k+1}_{Y,x},\]
which are respective right inverses of the maps
$i_{Z,k}:S^k(\mathfrak{m}_{Y,x}/\mathfrak{m}^2_{Y,x})\to S^k(\mathfrak{m}_{Z,x}/\mathfrak{m}^2_{Z,x})$
and $\tilde{i}_{Z,k}:\mathfrak{m}^k_{Y,x}/\mathfrak{m}^{k+1}_{Y,x}\to \mathfrak{m}^k_{Z,x}/\mathfrak{m}^{k+1}_{Z,x}$
induced by the inclusion $i:Z\hookrightarrow Y$. Moreover the diagrams
\[
\begin{array}{ccc}
S^k(\mathfrak{m}_{Z,x}/\mathfrak{m}^2_{Z,x}) & \stackrel{\varphi_{Z,k}}{\longrightarrow} & \mathfrak{m}^k_{Z,x}/\mathfrak{m}^{k+1}_{Z,x} \\
\mbox{\scriptsize $r_{Z,k}$}\downarrow & & \downarrow\mbox{\scriptsize $\tilde{r}_{Z,k}$} \\
S^k(\mathfrak{m}_{Y,x}/\mathfrak{m}^2_{Y,x}) & \stackrel{\varphi_{Y,k}}{\longrightarrow} & \mathfrak{m}^k_{Y,x}/\mathfrak{m}^{k+1}_{Y,x}
\end{array}
\quad\mbox{and}\quad
\begin{array}{ccc}
S^k(\mathfrak{m}_{Y,x}/\mathfrak{m}^2_{Y,x}) & \stackrel{\varphi_{Y,k}}{\longrightarrow} & \mathfrak{m}^k_{Y,x}/\mathfrak{m}^{k+1}_{Y,x} \\
\mbox{\scriptsize $i_{Z,k}$}\downarrow & & \downarrow\mbox{\scriptsize $\tilde{i}_{Z,k}$} \\
S^k(\mathfrak{m}_{Z,x}/\mathfrak{m}^2_{Z,x}) & \stackrel{\varphi_{Z,k}}{\longrightarrow} & \mathfrak{m}^k_{Z,x}/\mathfrak{m}^{k+1}_{Z,x}
\end{array}
\]
are commutative, where $\varphi_{Z,k}$ and $\varphi_{Y,k}$ are defined in a natural way.

In the setting of the lemma, for every $n\geq 1$, we denote $\mathfrak{m}_{n,x}:=\mathfrak{m}_{X_n,x}$. The retraction $r_n:X_{n+1}\to X_n$ induces maps $r_{n,k}:S^k(\mathfrak{m}_{n,x}/\mathfrak{m}_{n,x}^2)\to S^k(\mathfrak{m}_{n+1,x}/\mathfrak{m}_{n+1,x}^2)$ and $\tilde{r}_{n,k}:\mathfrak{m}_{n,x}^k/\mathfrak{m}_{n,x}^{k+1}\to \mathfrak{m}_{n+1,x}^k/\mathfrak{m}_{n+1,x}^{k+1}$, which are respective right inverses of the maps $i_{n,k}:S^k(\mathfrak{m}_{n+1,x}/\mathfrak{m}_{n+1,x}^2)\to S^k(\mathfrak{m}_{n,x}/\mathfrak{m}_{n,x}^2)$
and $\tilde{i}_{n,k}:\mathfrak{m}_{n+1,x}^k/\mathfrak{m}_{n+1,x}^{k+1}\to \mathfrak{m}_{n,x}^k/\mathfrak{m}_{n,x}^{k+1}$ induced by the inclusion $X_n\subset X_{n+1}$. Moreover the diagrams
\[
\begin{array}{ccc}
S^k(\mathfrak{m}_{n,x}/\mathfrak{m}^2_{n,x}) & \stackrel{\varphi_{n,k}}{\longrightarrow} & \mathfrak{m}^k_{n,x}/\mathfrak{m}^{k+1}_{n,x} \\
\mbox{\scriptsize $r_{n,k}$}\downarrow & & \downarrow\mbox{\scriptsize $\tilde{r}_{n,k}$} \\
S^k(\mathfrak{m}_{n+1,x}/\mathfrak{m}^2_{n+1,x}) & \stackrel{\varphi_{n+1,k}}{\longrightarrow} & \mathfrak{m}^k_{n+1,x}/\mathfrak{m}^{k+1}_{n+1,x}
\end{array}
\quad\mbox{and}\quad
\begin{array}{ccc}
S^k(\mathfrak{m}_{n+1,x}/\mathfrak{m}^2_{n+1,x}) & \stackrel{\varphi_{n+1,k}}{\longrightarrow} & \mathfrak{m}^k_{n+1,x}/\mathfrak{m}^{k+1}_{n+1,x} \\
\mbox{\scriptsize $i_{n,k}$}\downarrow & & \downarrow\mbox{\scriptsize $\tilde{i}_{n,k}$} \\
S^k(\mathfrak{m}_{n,x}/\mathfrak{m}^2_{n,x}) & \stackrel{\varphi_{n,k}}{\longrightarrow} & \mathfrak{m}^k_{n,x}/\mathfrak{m}^{k+1}_{n,x}
\end{array}
\]
commute, where $\varphi_{n,k}=\varphi_{X_n,k}$ (see also (\ref{**})).

Since $x\in X_{n_0}$ is singular,
there is $k\geq 2$ such that the map
$\varphi_{n_0,k}:S^k(\mathfrak{m}_{n_0,x}/\mathfrak{m}^2_{n_0,x})\to\mathfrak{m}^k_{n_0,x}/\mathfrak{m}^{k+1}_{n_0,x}$
is not injective, i.e., there is $a_{n_0}\in \ker\varphi_{n_0,k}\setminus\{0\}$.
We define the sequence $\{a_n\}$ by letting
\[
a_n=i_{n,k}\circ\cdots\circ i_{n_0-2,k}\circ i_{n_0-1,k}(a_{n_0})\ \mbox{ if $1\leq n\leq n_0$}\quad
\mbox{and}
\quad a_n=r_{n-1,k}\circ\cdots\circ r_{n_0+1,k}\circ r_{n_0,k}(a_{n_0})\ \mbox{ if $n\geq n_0$}.
\]
Then $a_n\in S^k(\mathfrak{m}_{n,x}/\mathfrak{m}^2_{n,x})$ and $i_{n,k}(a_{n+1})=a_n$ for all $n\geq 1$. Thus the sequence $a:=\{a_n\}$ is an element of the inverse limit $\lim\limits_{\leftarrow}S^k(\mathfrak{m}_{n,x}/\mathfrak{m}^2_{n,x})$. Moreover, we have $a\in\ker \hat\varphi_k\setminus\{0\}$, where $\hat\varphi_k:=\lim\limits_\leftarrow\varphi_{n,k}$. Therefore $\hat\varphi_k$ is not injective, and so $x$ is a singular point of $\mathbf{X}$.
\end{proof}

\subsection{Smoothness criterion for Schubert ind-varieties}
\label{section-6-2}

Let $\mathbf{G}=\mathbf{G}(E)$ (resp., $\mathbf{G}=\mathbf{G}^\omega(E)$).

Let $(A,\preceq_A)$ be a totally ordered set (resp., equipped with an anti-automorphism $i_A$).
A surjective map
 $\sigma:E\to A$ (resp., such that $i_A\circ \sigma=\sigma\circ i_E$)
gives rise to an $E$-compatible generalized flag $\mathcal{F}_\sigma=\{F'_{\sigma,\alpha},F''_{\sigma,\alpha}\}_{\alpha\in A}$ (see (\ref{relation-8})) and to the corresponding ind-variety $\mathbf{X}=\mathbf{Fl}(\mathcal{F}_\sigma,E)$ (resp., $\mathbf{X}=\mathbf{Fl}(\mathcal{F}_\sigma,\omega,E)$) (see Section \ref{section-3}). We consider the Schubert cell $\mathbf{B}\mathcal{F}_\sigma\subset\mathbf{X}$.
We denote its closure in $\mathbf{X}$ by $\mathbf{X}_\sigma$ (resp., $\mathbf{X}_\sigma^\omega$) and call it {\em Schubert ind-variety}.
Note that $\mathbf{X}_\sigma$ and $\mathbf{X}_\sigma^\omega$ depend on the choice of the splitting Borel subgroup $\mathbf{B}\subset\mathbf{G}$.

By Theorems \ref{theorem-1}\,{\rm (c)},\,{\rm (d)} and \ref{theorem-2}\,{\rm (c)},\,{\rm (d)}, the Schubert ind-variety $\mathbf{X}_\sigma$ (resp., $\mathbf{X}_\sigma^\omega$) admits a cell decomposition into Schubert cells $\mathbf{B}\mathcal{F}_\tau$ for $\tau\leq\sigma$ (resp., $\tau\leq_\omega\sigma$).

If $I\subset E$ is a finite subset, then the (finite-dimensional) flag variety $\mathrm{Fl}(\mathcal{F}_\sigma,I)$ (defined in Section \ref{section-3-3}) embeds in a natural way in the ind-variety $\mathbf{Fl}(\mathcal{F}_\sigma,E)$. The intersection $X_{\sigma,I}:=\mathbf{X}_\sigma\cap\mathrm{Fl}(\mathcal{F}_\sigma,I)$
is a Schubert variety in the usual sense. In the case of $\mathbf{G}=\mathbf{G}^\omega(E)$, if the subset $I\subset E$ is $i_E$-stable, the flag variety $\mathrm{Fl}(\mathcal{F}_\sigma,\omega,I)$ embeds in the ind-variety $\mathbf{Fl}(\mathcal{F}_\sigma,\omega,E)$. Again, the intersection $X_{\sigma,I}^\omega:=\mathbf{X}_\sigma^\omega\cap\mathrm{Fl}(\mathcal{F}_\sigma,\omega,I)$ is a Schubert variety in the usual sense.

Note that the Schubert ind-variety $\mathbf{X}_\sigma$ depends on the generalized flag $\mathcal{F}_\sigma$ and on the splitting Borel subgroup $\mathbf{B}$. Recall that $\mathbf{B}$ is the stabilizer of a maximal generalized flag  $\mathcal{F}_0$ (see Propositions \ref{proposition-1}, \ref{proposition-3-new}). 
Our singularity criterion (Theorem \ref{theorem-4} below) requires a technical assumption on $\mathbf{B}$ and $\mathcal{F}_\sigma$:
\begin{itemize}
\item[{\rm (H)}] At least one of the following conditions holds:
\begin{itemize}
\item[\rm (i)] $\mathcal{F}_0$ is a flag (i.e., $(\mathcal{F}_0,\subset)$ is isomorphic as ordered set to a subset of $(\mathbb{Z},\leq)$);
\item[\rm (ii)] $\mathcal{F}_{\sigma}$ is a flag, and $\dim F''_{\sigma,\alpha}/F'_{\sigma,\alpha}$ is finite whenever $0\not=F'_{\sigma,\alpha}\subset F''_{\sigma,\alpha}\not=V$.
\end{itemize}
\end{itemize} 

By $\mathrm{Sing}(X)$ we denote the set of singular points of a variety or an ind-variety $X$.

\begin{theorem}
\label{theorem-4}
Let $\mathbf{G}=\mathbf{G}(E)$ (resp., $\mathbf{G}=\mathbf{G}^\omega(E)$). Let $\sigma,\mathbf{X}_\sigma,\mathbf{X}_\sigma^\omega,X_{\sigma,I},X_{\sigma,I}^\omega$ be as above. Assume that hypothesis {\rm (H)} holds.
The following alternative holds: either
\begin{itemize}
\item[\rm (i)]
the variety $X_{\sigma,I}$ (resp., $X_{\sigma,I}^\omega$) is smooth
for all (resp., $i_E$-stable) finite subsets $I\subset E$; then the ind-variety $\mathbf{X}_\sigma$ (resp., $\mathbf{X}_\sigma^\omega$) is smooth;
\end{itemize}
or
\begin{itemize}
\item[\rm (ii)]
there is a finite subset $I_0\subset E$ such that, for every (resp., $i_E$-stable) finite subset $I\subset E$ with $I\supset I_0$, the variety $X_{\sigma,I}$ (resp., $X_{\sigma,I}^\omega$) is singular; then $\mathbf{X}_\sigma$ (resp., $\mathbf{X}_{\sigma}^\omega$) is singular and
\[\mathrm{Sing}(\mathbf{X}_\sigma)=\bigcup_{I\supset I_0}\mathrm{Sing}(X_{\sigma,I})
\qquad \mbox{(resp., \ $\displaystyle\mathrm{Sing}(\mathbf{X}_\sigma^\omega)=\bigcup_{I\supset I_0,\,\mbox{\scriptsize $i_E$-stable}}\mathrm{Sing}(X_{\sigma,I}^\omega)$).}\]
\end{itemize}
\end{theorem}

\begin{proof}
We provide the proof only for the case $\mathbf{G}=\mathbf{G}(E)$
(the proof in the case of $\mathbf{G}=\mathbf{G}^\omega(E)$ follows
the same scheme).

We need preliminary constructions and notation.
For a finite subset $I\subset E$ and an element $\tau\in
W(I)\cdot\sigma$, we define closed subgroups of $G(I)$ and $B(I)$ by letting
\begin{eqnarray*}
 & G_\tau(I):=\{g\in G(I):g(e)-e\in\langle e'\in I:\tau(e')\succ_A\tau(e)\rangle\ \ \forall e\in E\}\\
\mbox{and}
& B_\tau(I):=\{g\in G(I):g(e)-e\in\langle e'\in I:e'\prec_\mathbf{B}e,\ \tau(e')\succ_A\tau(e)\rangle\ \ \forall e\in E\}=B(I)\cap G_\tau(I).
\end{eqnarray*}
It is well known that the set
\[U_\tau(I):=\{g\mathcal{F}_\tau:g\in G_\tau(I)\}\]
is an open subvariety of $\mathrm{Fl}(\mathcal{F}_\sigma,I)$,
and the maps
\[\Phi_\tau:G_\tau(I)\to U_\tau(I),\ g\mapsto g\mathcal{F}_\tau\quad\mbox{and}\quad \Phi_\tau'=\Phi_\tau|_{B_\tau(I)}:B_\tau(I)\to B(I)\mathcal{F}_\tau\]
are isomorphisms of algebraic varieties. Thus, for every
$\tau\in\mathbf{W}(E)\cdot\sigma$, we obtain an open ind-subvariety
of $\mathbf{Fl}(\mathcal{F}_\sigma,E)$ by letting
\[\mathbf{U}_\tau:=\bigcup_{I}U_\tau(I),\]
where the union is taken over finite subsets $I\subset E$ such that $\tau\in W(I)\cdot\sigma$.
Clearly $\mathbf{B}\mathcal{F}_\tau\subset\mathbf{U}_\tau$, hence by
Theorem \ref{theorem-1}\,{\rm (a)} the open subsets
$\mathbf{U}_\tau$ (for $\tau\in\mathbf{W}(E)\cdot\sigma$) cover the
ind-variety $\mathbf{Fl}(\mathcal{F}_\sigma,E)$.

Let $I,J\subset E$ be finite subsets such that $I\subset J$. Let
$\mathrm{Fl}(\mathcal{F}_\sigma,I)$,
$\mathrm{Fl}(\mathcal{F}_\sigma,J)$ be corresponding
finite-dimensional flag varieties, and let
$\iota_{I,J}:\mathrm{Fl}(\mathcal{F}_\sigma,I)\to
\mathrm{Fl}(\mathcal{F}_\sigma,J)$ be the embedding defined in
Section \ref{section-3-3}. As noted in Proposition
\ref{proposition-11-new},
 we have $\iota_{I,J}(B(I)\mathcal{F}_\sigma)\subset B(J)\mathcal{F}_\sigma$,
hence $\iota_{I,J}(X_{\sigma,I})\subset X_{\sigma,J}$.

Let $\tau\in W(I)\cdot\sigma$. The inclusion $G_\tau(I)\subset
G_\tau(J)$ holds. Moreover, using that $g(e)=e$ for all $g\in
G_\tau(I)$, all $e\in J\setminus I$, in view of the definition of
the map $\iota_{I,J}$, we have
$\iota_{I,J}(g\mathcal{F}_\tau)=g\mathcal{F}_\tau\in U_\tau(J)$ for
all $g\in G_\tau(I)$. Hence the map $\iota_{I,J}$ restricts to an
embedding $\iota'_{I,J}:U_{\tau}(I)\cap X_{\sigma,I}\to
U_{\tau}(J)\cap X_{\sigma,J}$.

\medskip
\noindent {\it Claim 1.} Let $I,J\subset E$ be finite subsets such that $I\subset J$ and let $\tau\in W(I)\cdot\sigma$. Then, $\iota_{I,J}'$ restricts to an embedding $U_\tau(I)\cap\mathrm{Sing}(X_{\sigma,I})\to U_\tau(J)\cap\mathrm{Sing}(X_{\sigma,J})$.

Let $H\subset G(J)$ be the torus formed by the elements $h\in G(J)$ such that $h(e)=e$ for all $e\in I$ and $h(e)\in\mathbb{K}^*e$ for all $e\in J\setminus I$.
The torus $H$ acts on $X_{\sigma,J}$. From \cite{Iversen}, it follows that $\mathrm{Sing}((X_{\sigma,J})^H)\subset \mathrm{Sing}(X_{\sigma,J})$, where $(X_{\sigma,J})^H\subset X_{\sigma,J}$ stands for the subset of $H$-fixed points.
On the other hand, 
it is easy to see that the equality $\iota'_{I,J}(U_\tau(I)\cap X_{\sigma,I})=U_\tau(J)\cap (X_{\sigma,J})^H$ holds.
Thereby, 
\[\iota'_{I,J}(U_\tau(I)\cap \mathrm{Sing}(X_{\sigma,I}))=U_\tau(J)\cap\mathrm{Sing}((X_{\sigma,J})^H)\subset\mathrm{Sing}(X_{\sigma,J}).\]
This shows Claim 1.

\medskip
\noindent {\it Claim 2.} 
Let $I,J\subset E$ be finite subsets such that $J=I\cup\{e_J\}$ and let $\tau\in W(I)\cdot\sigma$. Assume that at least one of the following conditions holds:
\begin{itemize}
\item[\rm (i)] $e_J\prec_\mathbf{B}e$ for all $e\in I$; 
\item[\rm (ii)] $e_J\succ_\mathbf{B}e$ for all $e\in I$;
\item[\rm (iii)] $\tau(e_J)\preceq_A\tau(e)$ for all $e\in I$;
\item[\rm (iv)] $\tau(e_J)\succeq_A\tau(e)$ for all $e\in I$.
\end{itemize}
Then the map $\iota'_{I,J}:U_{\tau}(I)\cap X_{\sigma,I}\to
U_{\tau}(J)\cap X_{\sigma,J}$ admits a left
inverse $r'_{I,J}:U_{\tau}(J)\cap X_{\sigma,J}\to U_{\tau}(I)\cap
X_{\sigma,I}$.

\smallskip
We write an element $g\in \mathbf{G}(E)$ as a matrix
$(g_{e',e})_{e',e\in E}$ such that $g(e)=\sum_{e'\in E}g_{e',e}e'$.
Let $G_\tau(J)\to G_\tau(J)$, $g\mapsto g'$ and 
$R_{I,J}:G_\tau(J)\to G_\tau(I)$, $g\mapsto \tilde{g}$ be the
maps defined by
\begin{equation}
\label{newnew-19}
g'_{e',e}=\left\{
\begin{array}{ll}
0 & \mbox{if $e\not=e'=e_J$} \\
g_{e',e} & \mbox{otherwise},
\end{array}
\right.
\quad\mbox{and}\quad
\tilde{g}_{e',e}=\left\{
\begin{array}{ll}

0 & \mbox{if $e\not=e'$ and $e_J\in\{e,e'\}$} \\
g_{e',e} & \mbox{otherwise}.
\end{array}
\right.\end{equation} 
The map $R_{I,J}$ induces a morphism of algebraic
varieties $r_{I,J}:U_\tau(J)\to U_\tau(I)$,
$g\mathcal{F}_\tau\mapsto\tilde{g}\mathcal{F}_\tau$. It is clear
that $\tilde{g}=g$ whenever $g\in G_\tau(I)$, hence
$r_{I,J}(\iota_{I,J}(\mathcal{G}))=\mathcal{G}$ whenever
$\mathcal{G}\in U_\tau(I)$.

We claim that
\begin{equation}
\label{19-newnew}
\mathcal{G}\in U_\tau(J)\cap \mathbf{X}_\sigma\Rightarrow r_{I,J}(\mathcal{G})\in\mathbf{X}_\sigma.
\end{equation}
Let $\mathcal{G}=g\mathcal{F}_\tau$ with $g\in G_\tau(J)$. Assume that $\mathcal{G}\in\mathbf{X}_\sigma$. We first check that
\begin{equation}
\label{newnew-condition-2}
\mathcal{G}':=g'\mathcal{F}_\tau\in\mathbf{X}_\sigma
\end{equation}
with $g'$ as in (\ref{newnew-19}).
We distinguish four cases depending on the conditions {\rm (i)}--{\rm (iv)} of Claim 2.
\begin{itemize}
\item
Assume that condition {\rm (i)} holds.
Let $\mathcal{F}_0=\{F'_{0,e},F''_{0,e}:e\in E\}$ be the maximal generalized flag corresponding to $\mathbf{B}$, i.e., $F'_{0,e}=\langle e':e'\prec_\mathbf{B}e\rangle$ and $F''_{0,e}=\langle e':e'\preceq_\mathbf{B}e\rangle$ (see Section \ref{section-4-1}). In view of condition {\rm (i)} and the definition of the map $g\mapsto g'$, for any $F\in \mathcal{F}_0$ and any linear combination $\sum_{e\in J}\lambda_ee\in \langle J\rangle$, we have
\[\sum_{e\in J}\lambda_eg(e)\in F\Rightarrow \sum_{e\in J}\lambda_eg'(e)\in F.\]
This implication yields $\dim g'(M)\cap\langle J\rangle\cap F\geq \dim g(M)\cap\langle J\rangle\cap F$ for all $M\in\mathcal{F}_\tau$, all $F\in\mathcal{F}_0$.
It is well known that this property implies $g'\mathcal{F}_\tau\in\overline{B(J)g\mathcal{F}_\tau}\subset\mathbf{X}_\sigma$ (see, e.g., \cite{Billey-Lakshmibai}).
\item Assume that condition {\rm (ii)} holds. 
Then every $\mathcal{F}=\{F'_\alpha,F''_\alpha\}_{\alpha\in A}\in B(J)\mathcal{F}_\sigma$ satisfies $F''_\alpha\subset \langle E\setminus\{e_J\}\rangle$
whenever $\alpha\prec_A\sigma(e_J)$. The same property holds whenever $\mathcal{F}\in \overline{B(J)\mathcal{F}_\sigma}=\mathrm{Fl}(\mathcal{F}_\sigma,J)\cap\mathbf{X}_\sigma$.
Applying this observation to $\mathcal{F}=g\mathcal{F}_\tau$ (and noting that $\tau(e_J)=\sigma(e_J)$ because $\tau\in W(I)\cdot\sigma$), we deduce that $g_{e_J,e}=0$ for all $e\not=e_J$, whence $g'=g$. This clearly yields (\ref{newnew-condition-2}) in this case.
\item Assume that condition {\rm (iii)} holds. Then the definition of $G_\tau(J)$ yields $g_{e_J,e}=0$ for all $e\in I$, whence $g'=g$. This implies (\ref{newnew-condition-2}).
\item Finally, assume that condition {\rm (iv)} holds. Then the definition of $G_\tau(J)$ implies that $g(e_J)=e_J$. 
For $t\in\mathbb{K}^*$, let $\tilde{h}_t\in \mathbf{H}(E)$ be defined by
\begin{equation}
\label{newnew-22}
\tilde{h}_t(e)=\left\{\begin{array}{ll}
e & \mbox{if $e\not=e_J$} \\
te_J & \mbox{if $e=e_J$}
\end{array} \quad\mbox{for all $e\in E$.}\right.
\end{equation}
We have $g'\mathcal{F}_\tau=\lim_{t\to 0}\tilde{h}_tg\mathcal{F}_\tau$.
Since $\tilde{h}_tg\mathcal{F}_\tau\in\mathbf{X}_\sigma$ for all $t\in\mathbb{K}^*$, we get $g'\mathcal{F}_\tau\in\mathbf{X}_\sigma$, whence (\ref{newnew-condition-2}).
\end{itemize}
Therefore (\ref{newnew-condition-2}) holds true in all the cases.
Moreover, we have
\[\tilde{g}\mathcal{F}_\tau=\lim_{t\to \infty}\tilde{h}_{t}g'\mathcal{F}_\tau\]
with $\tilde{h}_t$ as in (\ref{newnew-22}).
Since $g'\mathcal{F}_\tau\in\mathbf{X}_\sigma$ (by (\ref{newnew-condition-2})) and $\tilde{h}_t$ stabilizes $\mathbf{X}_\sigma$, we conclude that $r_{I,J}(\mathcal{G})=\tilde{g}\mathcal{F}_\tau\in\mathbf{X}_\sigma$.
Whence (\ref{19-newnew}).

By (\ref{19-newnew}),
the map $r'_{I,J}:U_{\tau}(J)\cap
X_{\sigma,J}\to U_{\tau}(I)\cap X_{\sigma,I}$ obtained by
restriction of $r_{I,J}$ is well defined and fulfills the conditions of Claim 1.

\medskip
Relying on Claims 1 and 2, the proof of the theorem is carried out as
follows. If
$X_{\sigma,I}$
is smooth for all finite subsets $I\subset E$, then Lemma
\ref{lemma-4} guarantees that $\mathbf{X}_\sigma$ is a smooth
ind-variety. We now assume that there is a finite subset $I_0\subset
E$ such that $X_{\sigma,I_0}$ is singular. In this case
Lemma \ref{lemma-4} yields an inclusion
\[\mathrm{Sing}(\mathbf{X}_\sigma)\subset\bigcup_{I\supset I_0}\mathrm{Sing}(X_{\sigma,I})\]
where the union is taken over all finite subsets $I\subset E$ such
that $I\supset I_0$. For completing the proof it is sufficient to
prove that
\begin{equation}
\label{relation-1-last}
\mathrm{Sing}(X_{\sigma,I})\subset\mathrm{Sing}(\mathbf{X}_\sigma)
\end{equation}
for each finite subset $I\subset E$ with $I\supset I_0$. To show
this, let $\mathcal{G}\in\mathrm{Sing}(X_{\sigma,I})$. There is
$\tau\in W(I)\cdot\sigma$ such that $\mathcal{G}\in U_{\tau}(I)$.
We consider the two cases involved in assumption {\rm (H)}.
\begin{itemize}
\item If {\rm (H)\,(i)} holds, then let $e_0=\min I$ and $e_1=\max I$ (for the order $\preceq_\mathbf{B}$), and set $I'=\{e\in E:e_0\preceq_\mathbf{B}e\preceq_\mathbf{B}e_1\}$.
The set $I'$ is finite (by {\rm (H)\,(i)}). Moreover, again relying on {\rm (H)\,(i)}, we can find a filtration $E=\bigcup_{n\geq 1}E_n$ with $E_1=I'$ and $E_n=E_{n-1}\cup\{e_n\}$ for all $n\geq 2$, where $e_n$ is either the minimum or the maximum of $(E_n,\preceq_\mathbf{B})$.
\item If {\rm (H)\,(ii)} holds, then let $\alpha_0=\min\{\tau(e):e\in I\}$ and $\alpha_1=\max\{\tau(e):e\in I\}$ (for the order $\preceq_A$), and set $I'=I\cup \{e\in E:\alpha_0\prec_A\tau(e)\prec_A\alpha_1\}$.
The first part of  {\rm (H)\,(ii)} ensures that there are at most finitely many $\alpha\in A$ such that 
$\alpha_0\prec_A\alpha\prec_A\alpha_1$, while the second part of {\rm (H)\,(ii)} (together with the fact that $\tau\in \mathbf{W}(E)\cdot\sigma$) implies that $\tau^{-1}(\alpha)$ is finite for each such $\alpha$, hence the set $I'$ is finite.
Again relying on {\rm (H)\,(ii)}, we can construct a filtration $E=\bigcup_{n\geq 1}E_n$ with $E_1=I'$ and $E_n=E_{n-1}\cup\{e_n\}$ for all $n\geq 2$, where $e_n$ satisfies either $\tau(e_n)\preceq_A\tau(e)$ for all $e\in E_{n-1}$ or $\tau(e_n)\succeq_A\tau(e)$ for all $e\in E_{n-1}$.
\end{itemize}
In both cases, we get a filtration
$\{E_n\}_{n\geq 1}$ of  $E$ by finite subsets
such that $I\subset E_1$ and, for every $n\geq 2$, the pair $(E_{n-1},E_n)$ satisfies one of the conditions {\rm (i)}--{\rm (iv)} of Claim 2. We obtain an exhaustion of the open subset
$\mathbf{U}_{\tau}\cap\mathbf{X}_\sigma$ of $\mathbf{X}_\sigma$
given by the chain
\[U_{\sigma,\tau,1}\stackrel{\iota_{1}}{\hookrightarrow} U_{\sigma,\tau,2}\stackrel{\iota_{2}}{\hookrightarrow} U_{\sigma,\tau,3}\hookrightarrow\ldots\hookrightarrow U_{\sigma,\tau,n}\stackrel{\iota_n}{\hookrightarrow}\ldots\]
where $U_{\sigma,\tau,n}=U_{\tau}(E_n)\cap X_{\sigma,E_n}$ and
$\iota_n=\iota'_{E_n,E_{n+1}}$. Claim 1 implies that $\mathcal{G}$ is a
singular point of $U_{\sigma,\tau,1}$. By Claim 2,
 we can apply Lemma \ref{lemma-2} which implies that $\mathcal{G}$ is a singular point of $\mathbf{U}_\tau\cap\mathbf{X}_\sigma$, hence of $\mathbf{X}_\sigma$.
Therefore the inclusion  (\ref{relation-1-last}) holds. The proof
is complete.
%
\end{proof}

\begin{remark}
{\rm (a)} Note that hypothesis {\rm (H)} is valid in the case where $\mathbf{Fl}(\mathcal{F}_\sigma,E)$ is an ind-grassmannian. \\
{\rm (b)} Hypothesis {\rm (H)} is needed in the proof of Theorem \ref{theorem-4} for showing Claim 2 which is necessary for applying Lemma \ref{lemma-2}. We have no indication whatsoever that Theorem \ref{theorem-4} is not valid in general (without hypothesis {\rm (H)}).
\end{remark}

\begin{remark}
The Schubert ind-varieties $\mathbf{X}_\sigma$ considered in this paper form a narrower class than the ones considered by H.\ Salmasian \cite{Salmasian}. Indeed, a closed ind-subvariety $\mathbf{X}\subset\mathbf{Fl}(\mathcal{F},E)$ such that $\mathbf{X}\cap\mathrm{Fl}(\mathcal{F},I)$ is a Schubert variety for all finite subsets $I\subset E$ is a Schubert ind-variety in the sense of \cite{Salmasian}, and it may happen that $\mathbf{X}$ has no open $\mathbf{B}$-orbit and admits no smooth point in this case (see \cite[Section 2]{Salmasian}). On the other hand, the ind-variety $\mathbf{X}_\sigma$ defined in Section \ref{section-6-2} always contains the open $\mathbf{B}$-orbit $\mathbf{B}\mathcal{F}_\sigma$, and the points of $\mathbf{B}\mathcal{F}_\sigma$ are smooth in $\mathbf{X}_\sigma$.
\end{remark}

\subsection{Examples}

\label{section-4.2}

A consequence of Theorem \ref{theorem-4} is that the smoothness criteria for Schubert varieties of (finite-dimensional) flag varieties that are expressed in terms of pattern avoidance,
may pass to the limit at infinity.

For example, let us apply Theorem \ref{theorem-4} to the ind-variety $\mathbf{Fl}(\mathcal{F},E)$ for an $E$-compatible maximal generalized flag $\mathcal{F}$. In this case we have two total orders on the basis $E$: the first one $\preceq_\mathbf{B}$ corresponds to the splitting Borel subgroup $\mathbf{B}$, and the second order $\preceq_\mathcal{F}$ corresponds to the maximal generalized flag $\mathcal{F}$, i.e., $\mathcal{F}=\{F'_e,F''_e:e\in E\}$ is given by \[F'_e=\langle e'\in E:e'\prec_\mathcal{F}e\rangle,\quad F''_e=\langle e'\in E:e'\preceq_\mathcal{F}e\rangle.\]
By Theorem \ref{theorem-1}, the Schubert ind-varieties $\mathbf{X}_\sigma$ of $\mathbf{Fl}(\mathcal{F},E)$ are parametrized by the permutations $\sigma\in\mathbf{W}(E)$, and we have
\[\dim\mathbf{X}_\sigma=n_\mathrm{inv}(\sigma)=|\{(e,e')\in E:e\prec_\mathbf{B}e',\ \sigma(e')\prec_\mathcal{F}\sigma(e)\}|.\]
From Theorem \ref{theorem-4} and the known characterization of smooth Schubert varieties of full flag varieties in terms of pattern avoidance (see \cite[\S8]{Billey-Lakshmibai}) we obtain the following criterion.

\begin{corollary}
\label{corollary-2}
Assume that $\mathcal{F}$ or $\mathcal{F}_0$ is a flag, so that hypothesis {\rm (H)} is satisfied.
Let $\sigma\in\mathbf{W}(E)$. Then the Schubert ind-variety $\mathbf{X}_\sigma$ is singular if and only if there exist $e_1,e_2,e_3,e_4\in E$ such that $e_1\prec_\mathbf{B}e_2\prec_\mathbf{B}e_3\prec_\mathbf{B}e_4$ and ($\sigma(e_3)\prec_\mathcal{F}\sigma(e_4)\prec_\mathcal{F}\sigma(e_1)\prec_\mathcal{F}\sigma(e_2)$ or $\sigma(e_4)\prec_\mathcal{F}\sigma(e_2)\prec_\mathcal{F}\sigma(e_3)\prec_\mathcal{F}\sigma(e_1)$).
\end{corollary}

\begin{remark}
{\rm (a)}
Corollary \ref{corollary-2} shows in particular that, if the basis $E$ comprises infinitely many pairwise disjoint quadruples $(e_1,e_2,e_3,e_4)$ such that $e_1\prec_\mathbf{B}e_2\prec_\mathbf{B}e_3\prec_\mathbf{B}e_4$ and, say, $e_3\prec_\mathcal{F}e_4\prec_\mathcal{F}e_1\prec_\mathcal{F}e_2$, then for every permutation $\sigma\in\mathbf{W}(E)$, the Schubert ind-variety $\mathbf{X}_\sigma$ is singular. Thus, there exist pairs $(\mathbf{B},\mathcal{F})$ such that all Schubert ind-varieties of the ind-variety $\mathbf{Fl}(\mathcal{F},E)$ are singular. \\
{\rm (b)}
In the case where the ind-variety $\mathbf{Fl}(\mathcal{F},E)$ has finite-dimensional Schubert cells, it has one cell equal to a single point (see Theorem \ref{theorem-3}), hence has at least one smooth Schubert ind-variety.
Note that $\mathbf{Fl}(\mathcal{F},E)$ may have smooth Schubert ind-varieties although all its Schubert cells are infinite dimensional. Take for instance $E=\{e_i\}_{i\in\mathbb{Z}}$, let the order $\preceq_\mathbf{B}$ be the natural order on $\mathbb{Z}$, and let the order $\preceq_\mathcal{F}$ be the inverse order, i.e., $i\preceq_\mathcal{F}j$ if and only if $i\geq j$. Then every Schubert cell of $\mathbf{Fl}(\mathcal{F},E)$ is infinite dimensional, but the permutation $\sigma=\mathrm{id}_E\in\mathbf{W}(E)$ avoids the two forbidden patterns of Corollary \ref{corollary-2}, hence $\mathbf{X}_\sigma$ is smooth.
\end{remark}

As a second example,
we apply Theorem \ref{theorem-4} to the case of the ind-grassmannian $\mathbf{Gr}(2)$. In this case, for a splitting Borel subgroup $\mathbf{B}$, the Schubert ind-varieties $\mathbf{X}_\sigma$ are parametrized by the surjective maps $E\to \{1,2\}$ such that $|\sigma^{-1}(1)|=2$, or equivalently by the pairs of elements $\sigma=\{\sigma_1,\sigma_2\}\subset E$.
From Theorem \ref{theorem-4} and \cite[\S9.3.3]{Billey-Lakshmibai} we have:

\begin{corollary}
Let $\sigma=\{\sigma_1,\sigma_2\}\subset E$ with $\sigma_1\prec_\mathbf{B}\sigma_2$.
The Schubert ind-variety $\mathbf{X}_\sigma$ is smooth if and only if $\sigma_1$ is the smallest element of the ordered set $(E,\preceq_\mathbf{B})$ or $\sigma_1,\sigma_2$ are two consecutive elements of $(E,\preceq_\mathbf{B})$.
\end{corollary}

\section*{Acknowledgement}

This project was supported in part by the Priority Program ``Representation Theory'' of the DFG (SPP 1388).
L. Fresse acknowledges partial support through
ISF Grant Nr. 882/10 and ANR Grant Nr. ANR-12-PDOC-0031.
I. Penkov thanks the Mittag-Leffler Institute in Djursholm for its hospitality.

\end{document}